\documentclass[12pt,a4paper,oneside]{amsart}
\usepackage[utf8]{inputenc}
\usepackage{amsfonts, amsmath, amssymb, amsthm, amscd}
\usepackage{upgreek}
\usepackage{anysize}
\usepackage{mathtools}
\marginsize{2cm}{2cm}{1.5cm}{1.5cm}
\usepackage[english]{babel}
\usepackage{cmap}
\usepackage{enumerate}
\usepackage{xcolor}
\usepackage{hyperref}
\hypersetup{
	colorlinks   = true, 
	urlcolor     = blue, 
	linkcolor    = blue, 
	citecolor    = red 
}
\newcommand{\Href}[2]{\hyperref[#2]{#1~\ref{#2}}}
\usepackage{anysize}
\marginsize{1.8cm}{1.8cm}{1.5cm}{1.5cm}

\numberwithin{equation}{section}
\newtheorem{thm}{Theorem}[section]
\newtheorem{prp}{Proposition}[section]
\newtheorem{lemma}{Lemma}[section]
\newtheorem{cor}{Corollary}[section]
\newtheorem{claim}{Claim}[section]
\theoremstyle{definition}

\newcommand{\Red}{\Re^d}

\newcommand{\ball}[1]{\mathbf{B}^{#1}}
\newcommand{\norm}[1]{\left\|#1\right\|}

\newcommand{\conv}[1]{\mathrm{conv}\left(#1\right)}
\newcommand{\pos}{\mathrm{pos}}
\newcommand{\diag}{\mathrm{diag}}
\newcommand{\bd}[1]{\mathrm{bd}\left(#1\right)}

\newcommand{\inter}[1]{\mathrm{int}\left(#1\right)}
\newcommand{\iprod}[2]{\left\langle#1,#2\right\rangle}
\newcommand{\tr}[1]{\mathrm{trace}\left(#1\right)}
\newcommand{\length}[1]{\mathrm{length}\left(#1\right)}

\newcommand{\enorm}[1]{\left|#1\right|}

\def\R{{\mathbb R}}

\renewcommand{\Re}{\mathbb{R}}

\def\ellips{\mathcal{E}}%
\newcommand{\Redp}{\Re^{d+1}}

\newcommand{\st}{:\;}
\newcommand{\MM}{\mathcal{M}}

\def\phi{\varphi}
\def\epsilon{\varepsilon}
\def\alpha{\upalpha}

\newcommand{\loginfconv}[2]{#1 \star #2}%

\newcommand{\vol}[1]{\operatorname{vol}\nolimits_{#1}}%
\newcommand{\betaf}[2]{\operatorname{B} \!\left(#1, #2 \right)}
\newcommand{\gammaf}[1]{\operatorname{\Gamma} \!\left(#1 \right)}
\newcommand{\di}{\,\mathrm{d}}
\newcommand{\dsymm}{$d$-symmetric}
\newcommand{\bernardoell}{the AMJV ellipsoid
}%
\newcommand{\Bernardoell}{The AMJV ellipsoid
}%

\newcommand{\upthing}[1]{\overline{#1}}%
\newcommand{\sthing}[2][s]{\prescript{(#1)}{}{#2}}%
\newcommand{\slift}[2][s]{\sthing[#1]{\upthing{#2}}}%
\newcommand{\smeasure}[2][s]{\sthing[#1]{\mu}\!\left(#2\right)}%
\newcommand{\smarg}[2][s]{\sthing[#1]{\mathrm{marginal}\left({#2}\right)}}%
\newcommand{\volbs}[1][s]{\sthing[#1]{\kappa}_{d+1}}%
\newcommand{\sellbody}[2][s]{\upthing{E}(#2,#1)}%
\newcommand{\selldense}[2][s]{\sthing[#1]{J}_{#2}}%
\newcommand{\tpdfs}{\texorpdfstring{$s$}{s}}%
\newcommand{\jsellipsoid}{John $s$-ellipsoid}
\newcommand{\jsellipsoids}{John $s$-ellipsoids}
\newcommand{\jsfunction}{John $s$-function}
\newcommand{\jsfunctions}{John $s$-functions}

\newcommand{\noshow}[1]{}

\newcommand{\irat}{\operatorname{I.rat}}%
\newcommand{\sintrat}[2][s]{{
\sthing[#1]{\irat}}{#2}}

\def\heightfunc{\hslash}
\newcommand{\shf}[1]
{\heightfunc_{#1}}%

\date{\today}

\title{Functional John Ellipsoids}
\author{Grigory Ivanov\address{Grigory Ivanov: Institute of Science and 
Technology Austria; Moscow Inst. of Physics and Technology, 
Moscow, Russia}
\email{grimivanov@gmail.com}
\and
M\'arton Nasz\'odi\address{M\'arton Nasz\'odi: Alfr\'ed R\'enyi Inst. of Math.; 
MTA-ELTE Lend\"ulet Combinatorial Geometry Research Group;
Dept. of Geometry, Lor\'and E\"otv\"os University, Budapest}
\email{marton.naszodi@math.elte.hu}
}
\makeatletter

\@namedef{subjclassname@2020}{%
  \textup{2020} Mathematics Subject Classification}
\makeatother

\begin{document}
\begin{abstract}
We introduce a new way of representing logarithmically concave functions on 
$\mathbb{R}^{d}$. It allows us to extend the notion of the largest volume 
ellipsoid contained in a convex body to the setting of logarithmically concave 
functions as follows. For every $s>0$, we define a class of non-negative 
functions on $\mathbb{R}^{d}$ derived from ellipsoids in $\mathbb{R}^{d+1}$. 
For 
any log-concave function $f$ on $\mathbb{R}^{d}$, and any fixed $s>0$, we 
consider functions belonging to this class, and find the one with the largest 
integral under the condition that it is pointwise less than or equal to $f$, 
and 
we call it the \emph{\jsfunction} of $f$. After establishing existence and 
uniqueness, we give a characterization of this function similar to the one 
given 
by John in his fundamental theorem. We find that \jsfunctions{} converge to 
characteristic functions of ellipsoids as $s$ tends to zero and to Gaussian 
densities as $s$ tends to infinity.

As an application, we prove a quantitative Helly type result: the integral of 
the pointwise minimum of any family of log-concave functions is at least a 
constant $c_d$ multiple of the integral of the pointwise minimum of a properly 
chosen subfamily of size $3d+2$, where $c_d$ depends only on $d$.
\end{abstract}

\maketitle


\section{Main results and the structure of the paper}\label{sec:intro}


The largest volume ellipsoid contained in a convex body 
in $\mathbb{R}^{d}$ and, in particular, John's result \cite{Jo48} 
characterizing 
it, plays a fundamental role in convexity. The latter states that the 
origin-centered Euclidean unit ball is the largest volume ellipsoid contained 
in 
the convex body $K$ if and only if it is contained in $K$ and the contact 
points (that is, the intersection points of the unit sphere and the boundary of 
$K$) satisfy a certain algebraic condition.

Alonso-Guti{\'e}rrez, Gonzales Merino, Jim{\'e}nez and Villa 
\cite{alonso2018john} 
extended the notion of the John ellipsoid to the setting of logarithmically 
concave functions. To any 
log-concave function $f$ of finite positive integral on $\Red$, they associate 
an ellipsoid in $\Red$, which we call 
\emph{\bernardoell{}}, in the following manner.

We denote the $L_{\infty}$ norm of $f$ by  $\norm{f}$. For every $\norm{f} > 
\beta > 0$, consider the superlevel set $\{x\in\Red\st f(x)\geq\beta\}$ of $f$. 
This is a bounded convex set with non-empty interior, we take its largest 
volume 
ellipsoid, and multiply the volume of this ellipsoid by $\beta$. As shown in  
\cite{alonso2018john}, there is a unique height $\beta_0 \in \left[0, 
\norm{f}\right]$ such that this product is maximal. \Bernardoell{} is the 
ellipsoid $E$ in $\Red$ obtained for this $\beta_0$. 

We propose an alternative route to this extension with the introduction of a 
parameter $s>0$ that can be chosen arbitrarily. As a limit as $s$ tends to 
zero, 
we recover the above described approach of Alonso-Guti{\'e}rrez, Gonzales 
Merino, 
Jim{\'e}nez and Villa. The main advantage of our framework is that it implies a 
John type characterization of the maximal ellipsoid. We present an application 
of this characterization: a quantitative Helly type result for the integral of 
the pointwise minimum of a family of logarithmically concave functions. 

The paper is organized as follows.

In \Href{Section}{sec:smeasure}, we introduce the notions of $s$-lifting and 
$s$-volume, which will frame our study of logarithmically concave functions, 
and 
then, we define our main object of interest, the \emph{\jsellipsoid} (an 
ellipsoid in $\Redp$) and the \emph{\jsfunction} (a function on $\Red$) of a 
log-concave function $f$ on $\Red$.

The idea is the following. Fix an $s>0$ and consider the graph of the function 
$f^{1/s}$, which is a set in $\Redp$, and turn it into a not necessarily convex 
body in $\Redp$, which we call the \emph{$s$-lifting} of $f$. We define also a 
measure-like quantity, the \emph{$s$-volume} of sets in $\Redp$. Then we look 
for the ellipsoid in $\Redp$ which is contained in the $s$-lifting of $f$ and 
is 
of maximal $s$-volume. We call this ellipsoid in $\Redp$ the \jsellipsoid{} of 
$f$. This ellipsoid defines a function on $\Red$, which is the \jsfunction{} of 
$f$. This function is pointwise 
less than or equal to $f$.

In \Href{Subsection}{sec:motivationagain}, we describe our definitions in 
geometric terms and in \Href{Subsection}{sec:height}, in terms of a functional 
optimization problem, concluding the second introductory section.

In \Href{Section}{sec:basicbounds}, we prove some 
basic inequalities about the quantities introduced before. As an immediate 
application of these inequalities, we 
obtain a compactness result that, in the next section, yields that the 
\jsellipsoid{} exists.

\Href{Section}{sec:interpolation} contains one of our main tools, 
\emph{interpolation between ellipsoids}.
In the classical theory of the John ellipsoid, the uniqueness of the largest 
volume ellipsoid contained in a convex body $K$ in $\Red$ may be proved in the 
following way. Assume that $E_1=A_1\ball{d}+a_1$ and $E_2=A_2\ball{d}+a_2$ are 
ellipsoids of the same volume contained in $K$, where $\ball{d}$ denotes the 
Euclidean unit ball, $A_1, A_2$ are matrices, and $a_1, a_2\in\Red$. Then the 
ellipsoid $\frac{A_1+A_2}{2}\ball{d}+\frac{a_1+a_2}{2}$ is also contained in 
$K$ and its volume is larger than that of $E_1$ and $E_2$. 

One cannot apply this argument in our setting in a straightforward 
manner, as the set we consider is not convex. However,  we show that if 
two ellipsoids in $\Redp$ of the same $s$-volume are contained in the 
$s$-lifting of a log-concave function $f$, then one can define a third 
ellipsoid ``between'' the two ellipsoids which is of larger $s$-volume. This 
intermediate ellipsoid is obtained as a non-linear combination of the 
parameters determining the two ellipsoids.

As an immediate application, we obtain that the \jsellipsoid{} is 
unique, see \Href{Theorem}{thm:johnunicity}.

In \Href{Section}{sec:johncond}, we state and prove 
a necessary and sufficient condition for the $(d+1)$-dimensional Euclidean 
unit ball $\ball{d+1}$ to be the \jsellipsoid{} of a log-concave function 
$f$ on $\Red$, see \Href{Theorem}{thm:johncond}. Here, we phrase a simplified
version of it.

\begin{thm}\label{thm:johncondbasic}
Let $\upthing{K}=\{(x,\xi)\in\Redp\st |\xi|\leq f(x)/2 \}\subseteq \Redp$ 
denote the symmetrized subgraph of an upper semi-continuous log-concave 
function 
$f$ on $\Red$ of positive non-zero integral. Assume that the 
$(d+1)$-dimensional Euclidean unit ball $\ball{d+1}$ is contained in 
$\upthing{K}$. Then the following are equivalent.
\begin{enumerate}
 \item
The ball $\ball{d+1}$ is the unique maximum volume ellipsoid contained in 
$\upthing{K}$. 
 \item
There are contact points 
$\upthing{u}_1,\ldots,\upthing{u}_k\in\bd{\ball{d+1}}\cap\bd{\upthing{K}}$, and 
positive weights $c_1,\ldots,c_k$ such that 
\begin{equation*}
 \sum_{i=1}^k c_i \upthing{u}_i\otimes \upthing{u}_i =\upthing{I}
 \;\;\;\mbox{ and }\;\;\;\;
 \sum_{i=1}^k c_i u_i =0,
\end{equation*}
where $u_i$ is the orthogonal projection of $\upthing{u}_i$ onto $\Red$ and 
$\upthing{I}$ is the $(d+1)\times(d+1)$ identity matrix.
\end{enumerate}
\end{thm}

The implication from (1) to (2) is proved in more or less the same way as 
John's 
fundamental theorem about convex bodies, there are hardly any additional 
difficulties. The converse however, is not straightforward, since 
$\upthing{K}$ is not a convex body in general. That part of the proof relies 
heavily on the technique of interpolation between ellipsoids described in 
\Href{Section}{sec:interpolation}.

We note that non-convex sets in place of ellipsoids in a similar context for 
sets (not functions) were considered in \cite{BR02}. In our case, however, 
it is the set which contains the other (the ``container set'') which is 
non-convex, 
and that is the source of difficulties in finding the optimum (maximum volume 
or 
integral).

We give also an equialent, purely functional formulation of 
\Href{Theorem}{thm:johncondbasic} without reference to bodies in 
$(d+1)$-dimensional space, see \Href{Theorem}{thm:johncondfunc}.

In \Href{Section}{sec:furtherIneq}, we describe the relationship 
between the approach of Alonso-Guti{\'e}rrez, Gonzales Merino, Jim{\'e}nez and 
Villa 
\cite{alonso2018john} and our approach.

In \Href{Theorem}{thm:getbackbernardo}, we show that 
$\beta_0\chi_E$ is 
the 
limit (in a rather strong sense) of our \jsfunctions{} 
as $s$ tends to $0$, where $\beta_0$ is the height of \bernardoell{} $E$.

This result is based on the comparison of the $s$-volumes of \jsellipsoids{} for 
distinct values of $s$. We compare also these $s$-volumes and the integral of 
$f$ obtaining a bound on the integral ratio, the functional analogue of volume 
ratio.

In \Href{Section}{sec:sinfty}, we study the \jsfunctions{} as $s$ tends to 
infinity.
We show that the limit may only be a Gaussian density, see 
\Href{Theorem}{thm:s_infty_approx}. 
What is perhaps surprising is that the largest integral Gaussian density that is 
pointwise less than or equal to $f$ is not necessarily unique, see 
\Href{Section}{sec:non_unique_gaussian}. We show however, 
that in this case, the two Gaussians are translates of each other, see 
\Href{Theorem}{thm:infinity_ellipsoid}.

Finally, \Href{Section}{sec:BKP} contains the proof of our quantitative 
Helly type result. This is a non-trivial application of the results of the 
previous sections. We describe it in detail here.

For a positive integer $n$, we denote by $[n]$ the set $[n]=\{1,2,\ldots,n\}$.
For $m\leq n$, the family of subsets of $[n]$ of cardinality at most 
$m$ is denoted by $\binom{[n]}{\leq m}$.

According to Helly's theorem, \emph{if the intersection of a finite family of 
convex sets in $\Red$ is empty, then it has a subfamily of at most $d+1$ 
members such that the intersection of all members of the subfamily is empty.}

A quantitative variant of Helly's theorem was discovered by B\'ar\'any, 
Katchalski and Pach \cite{BKP82}, stating the following. \emph{
Let $K_1,\ldots,K_n$ be convex sets in $\Red$. Then there is 
a set $\sigma\in\binom{[n]}{\leq 2d}$ of at most $2d$ indices such 
that
\[
\vol{d}\left(\bigcap_{i\in\sigma}K_i\right) \leq 
c_d\vol{d}\left(\bigcap_{i\in[n]}K_i\right),
\]
where $c_d$ depends only on $d$.}

In \cite{BKP82}, it is shown that one can take $c_d=d^{2d^2}$ and it is 
conjectured that the theorem should hold with $c_d=d^{cd}$ for a proper 
absolute constant $c>0$. It was confirmed in \cite{Nas16} with $c_d\approx 
d^{2d}$, where it is also shown that such result will not hold with $c_d\ll 
d^{d/2}$. The argument in \cite{Nas16} was refined by Brazitikos \cite{Bra17} 
who showed that one may take $c_d\approx d^{3d/2}$. For more on quantitative 
Helly type results, see the surveys \cite{HW18survey,DGFM19survey}

Observe that the pointwise minimum of a family of log-concave 
functions is again log-concave. Our quantitative Helly type result is the 
following.

\newcommand{\hellyno}{3d+2}
\begin{thm}[]\label{thm:BKP}
Let $f_1,\ldots,f_n$ be upper semi-continuous log-concave functions  on $\Red$. 
For 
every $\sigma\subseteq [n]$, let $f_\sigma$ denote the pointwise minimum:
\begin{equation*}
 f_\sigma(x)=\min\{f_i(x)\st i\in \sigma\}.
\end{equation*}
Then there is a set $\sigma\in\binom{[n]}{\leq \hellyno}$ of at most $\hellyno$ 
indices such 
that, with the notation $f=f_{[n]}$, we have 
\begin{equation}\label{eq:BKPgoal}
\int_{\Red} f_\sigma\leq  \left(100 d\right)^{5d/2}\int_{\Red} f. 
\end{equation}
\end{thm}

The characteristic function of a convex set is log-concave, and pointwise 
minimum of functions corresponds to intersection of sets. Thus, 
\Href{Theorem}{thm:BKP} yields a quanitative Helly type result about convex sets 
as a special case. When comparing 
quanitative Helly type results, one may consider the Helly number and the bound 
on the volume 
(integral). Regarding the Helly number, on the one hand, we show in 
\Href{Subsection}{sec:BKPlowerbound} that in our functional case, it is at least 
$2d+1$, 
unlike in the case of convex sets, where it is $2d$.
Our bound on the integral is of the right order of magnitude, as it can not be 
improved beyond $d^{d/2}$ even for convex sets, see \cite{Nas16}.

At the expense of obtaining a much worse bound on the integral in place of the 
multiplicative constant $d^{5d/2}$, we can show a similar result with Helly 
number $2d+1$ instead of $3d+2$. That result will be part of a sequel to the 
present paper.

We note also that our proof of this functional result does not make use of the 
analogous statement for convex sets.

\subsection{Notation, Basic Terminology}\label{sec:notation}
We denote the Euclidean unit ball in $\Re^n$ by $\ball{n},$
and we
write $\enorm{\cdot}$ for the Euclidean norm.

We identify the hyperplane in $\Redp$ spanned by the first $d$ standard basis 
vectors with $\Red$.
A set $C\subset\Redp$ is \emph{\dsymm} if $C$ is symmetric about $\Red,$
that is, if $(2P-I)C=C$, where $P:\Redp\to\Redp$ is the orthogonal projection 
onto $\Red$.

For a square matrix $A\in \Re^{d\times d}$ and a scalar $\alpha \in \Re$, we 
denote by $A\oplus \alpha$ the $(d+1)\times(d+1)$ matrix 
\[
 A\oplus \alpha=\left(
 \begin{array}{cc}
  A&0\\
  0&\alpha
 \end{array}
 \right).
\]

For a function $f:\Red\to\Re$ and a scalar 
$\alpha\in\Re$, we denote the \emph{superlevel set} $\{x\in\Red\st 
f(x)\geq\alpha\}$ by 
$[f\geq\alpha]$.
The \emph{epigraph} of $f$ is the set 
$\operatorname{epi}(f)=\{(x,\xi)\in\Redp\st \xi\geq f(x)\}$ in $\Redp$. 
The $L_{\infty}$ norm of a function $f$ is denoted by $\norm{f}$.

We will say that a function $f_1: \R^d \to \R$  is \emph{below} a function 
$f_2: \R^d \to \R$, and denote it as $f_1\leq f_2$,
 if $f_1$ is pointwise less than or equal to $f_2$, that is, 
 $f_1 (x) \leq f_2 (x)$ for all $x \in \R^d.$

A function $\psi:\Red\to\Re\cup\{\infty\}$ is called \emph{convex} if 
$\psi((1-\lambda)x+\lambda y)\leq (1-\lambda)\psi(x)+\lambda\psi(y)$ for every 
$x,y\in\Red$ and $\lambda\in[0,1]$. 
A function $f$ on $\Red$ is \emph{logarithmically concave} (or 
\emph{log-concave} for short) if $f=e^{-\psi}$ for a convex function $\psi$ on 
$\Red$. We say that a log-concave function $f$ on $\Red$ is a \emph{proper 
log-concave function} if $f$ is upper semi-continuous and has finite positive 
integral.

We will use $\prec$ to denote the standard partial order on the cone of 
positive semi-definite matrices, that is,
we will write $A \prec B$ if $B - A$ is positive definite.
We recall the additive and the multiplicative form of 
\emph{Minkowski's determinant  inequality}. Let $A$ and $B$ be 
positive definite matrices of order $d$. Then, for any $\lambda \in (0,1),$
\begin{equation}
\label{eq:minkowski_det_ineq}
\left(\det\left( \lambda A + (1 - \lambda)B\right)\right)^{1/d} \geq
\lambda \left(\det A\right)^{1/d} + 
(1 -\lambda)\left(\det B\right)^{1/d},
\end{equation}
with equality if and only if $A = cB$ for some $c > 0;$ and 
\begin{equation}
\label{eq:minkowski_det_multipl_ineq}
\det\left( \lambda A + (1 - \lambda)B\right) \geq
\left(\det A\right)^{\lambda} \cdot \left(\det B\right)^{1 -\lambda},
\end{equation}
with equality if and only if $A = B.$
\section{The \tpdfs-volume, the \tpdfs-lifting and the 
\tpdfs-ellipsoids}\label{sec:smeasure}

\subsection{Motivation for the definitions}\label{sec:motivation}

One way to obtain a log-concave function $f$ on $\Red$ is to fix a convex body 
$\upthing{K}$ in $\Re^{d+s}$ for some positive integer $s$, take the 
\emph{uniform measure} on $\upthing{K}$ (that is, the absolutely continuous 
measure whose density is the characteristic function of $\upthing{K}$) and take 
the density of its marginal on $\Red$. Conversely, it is well known that any 
log-concave function is a limit of functions obtained this way. 
This representation of log-concave functions was used by Artstein-Avidan, 
Klartag and Milman in \cite{AKM04}, where a functional form of the Santal\'o 
inequality is proved.

If $f$ is obtained this way, then it is natural to consider the largest volume 
($(d+s)$-dimensional) ellipsoid contained in $\upthing{K}$, and take the 
uniform 
measure on this ellipsoid. The marginal on $\Red$ of this measure could be a 
candidate for the John ellipsoid function on $f$.
However, for a given $f$, the convex body $\upthing{K}$ in $\Re^{d+s}$ 
described above is not unique, if it exists. One may take the 
\emph{Schwarz symmetrization} of any such $\upthing{K}$ about $\Red$ (defined 
in 
\Href{Subsection}{sec:motivationagain}) to obtain a new convex body in 
$\Re^{d+s}$ which is now symmetric about $\Red$ and still has the property that 
the density of the marginal on $\Red$ of the uniform measure on it is $f$. 
Since 
the Schwarz symmetrization of an ellipsoid is again an ellipsoid, the John 
ellipsoid of the Schwarz symmetrization of $\upthing{K}$ is at least as large 
as 
the John ellipsoid of $\upthing{K}$. In summary, the marginal on 
$\Red$ of the uniform measure on the John ellipsoid of the Schwartz 
symmetrization of 
$\upthing{K}$ is a function of special form, and is below $f$. Moreover, it is 
of maximal integral among functions of this special form that are below $f$. 
This is now a good candidate for the John function of $f$.

With one more idea, we can reduce the dimension from $d+s$ to $d+1$. In fact, 
due to the symmetry about $\Red$, there is no need to consider a body in 
$\Re^{d+s}$. Instead, we may consider the section of this body by the linear 
subspace spanned by $\Red$ and any vector which is not in $\Red$, say 
$e_{d+1}$. 
We just need to remember that the last coordinate in $\Redp$ represents $s$ 
coordinates when it comes to computing the marginal of the uniform distribution 
of a convex body in $\Redp$.

In what follows, we formalize this reasoning without referring to any 
$(d+s)$-dimensional convex body. An advantage of the formalism that follows is 
that it works for non-integer $s$, as well as for any proper log-concave 
function $f$, and not only for functions obtained as the marginals of the 
uniform measure on some higher dimensional convex set.

We will mostly study objects in $\Red$ and in $\Redp$. For an easier reading, 
we 
emphasize that a set is in $\Redp$ by using a bar in its symbol, e.g. 
$\upthing{K}$.

\subsection{The \tpdfs-volume and its \tpdfs-marginal}

Fix a positive real $s$. For every $x\in\Red$, we denote the line in $\Redp$
perpendicular to $\Red$ at $x$ by $\ell_x$.

Let $\upthing{C}\subset\Redp$ be a \dsymm{} Borel set. 
The \emph{$s$-volume} of $\upthing{C}$ is defined by
\[\smeasure{\upthing{C}}=\int_{\Red} 
\left[\frac{1}{2}\length{\upthing{C}\cap\ell_x}\right]^s \di x.\] 

Note that $\smeasure{\cdot}$ is not a measure on $\Redp$. However, for 
any \dsymm{} Borel set $\upthing{C}$ in $\Redp$, the \emph{$s$-marginal} of 
$\upthing{C}$ on $\Red$ defined for any Borel set $B$ in $\Red$ by 
\begin{equation}\label{eq:marginaldef}
\smarg{\upthing{C}}(B)=\int_{B} 
\left[\frac{1}{2}\length{\upthing{C}\cap\ell_x}\right]^s \di x 
\end{equation}
is a measure on $\Red$.

We note that for any matrix $\upthing{A}=A\oplus\alpha$, where 
$A\in\Re^{d\times d}$ and $\alpha\in\Re$, any 
\dsymm{} set $\upthing{C}$ in $\Redp$ and any Borel set $B$ in $\Red$, we have
\begin{equation}\label{eq:lintrafo}
\left\{
\begin{array}{lcl}
\smarg{\upthing{A}\, \upthing{C}}(AB)&=&|\det A|\cdot |\alpha|^s\cdot 
\smarg{\upthing{C}}(B),\\ 
\smeasure{\upthing{A}\, \upthing{C}}&=&|\det A| \cdot|\alpha|^s 
\cdot\smeasure{\upthing{C}}.
\end{array}
\right.
\end{equation}


\subsection{The \tpdfs-lifting of a function}

Let $f:\Red\to [0, \infty)$ be a function and $s>0$. The \emph{$s$-lifting} of 
$f$ is a 
\dsymm{} set in $\Redp$ defined by
\[
\slift{f} = 
\left\{ (x,\xi) \in \Redp \st
 |\xi| \le 
\left(f(x)\right)^{1/s}\right\}.
\] 
Note the following \emph{scaling property} of $s$-lifting: for any 
$\gamma>0$, 
\begin{equation}\label{eq:scaleliftf}
\slift{(\gamma f)}=\left(I \oplus \gamma^{1/s}\right)\slift{f}. 
\end{equation}

Clearly, for any Borel set $B$ in $\Red$,
\begin{equation*}
 \int_B f = \smeasure{\slift{f}\cap \left(B\times \Re\right)},
\end{equation*}
that is,
$
 \smarg{\slift{f}} \text{ is the measure on }\Red\text{ with density }f.
$

\subsection{Ellipsoids}
Let $A$ be a positive definite matrix in $\Re^{d\times d}$ and $a\in\Red$.
They determine an \emph{ellipsoid} defined by
\begin{equation}\label{eq:ellipsoid_operator_duality}
A\left(\ball{d}\right)+a.
\end{equation}
Note that  $A\left(\ball{d}\right)+a=\{x \in \Red \st \iprod{A^{-1} x}{A^{-1} 
x} \le 1 \}+a$.

We will consider \dsymm{} ellipsoids in $\Redp$ (see 
\Href{Section}{sec:notation} for the definition of $d$-symmetry). To describe 
them, we introduce the vector space
\begin{equation}\label{eq:Mdef}
 \MM=\left\{(\upthing{A},a)\st \upthing{A}\in\Re^{(d+1)\times(d+1)},\; 
\upthing{A}^{\intercal}=\upthing{A},\; a\in\Red\right\},
\end{equation}
and the { convex} cone 
\begin{equation}\label{eq:Edef}
 \ellips=\left\{(A\oplus\alpha,a)\in\MM,\;A\in\Re^{d\times d} \text{ positive 
definite}, \alpha>0\right\}.
\end{equation}
Clearly, any \dsymm{} ellipsoid in $\Redp$ is represented by 
\[
 (A\oplus\alpha)\ball{d+1}+a,
\]
in a unique way. Thus, from this point on, we identify $\ellips$ with the set of 
all \dsymm{} ellipsoids in $\Redp$, 
and in particular, we may write $\smeasure{(A\oplus\alpha,a)}$ to refer to the 
$s$-volume of the corresponding ellipsoid.
We note that
\begin{equation}\label{eq:MMdim}
 \dim\MM=\frac{(d+1)(d+2)}{2}+d.
\end{equation}

\subsection{Definition of the John \tpdfs-ellipsoid of a function}
\label{sec:sellipsoid}
Fix $s>0$ and let $z(f,s)$ denote the supremum of the $s$-volumes of all 
\dsymm{} ellipsoids 
$\upthing{E}$ in $\Redp$ with $\upthing{E}\subseteq \slift{f}$. 
\Href{Lemma}{lem:boundedness} and 
a standard compactness argument yield that this supremum is attained. We will 
see (\Href{Theorem}{thm:johnunicity}) that it is attained on a unique 
ellipsoid. We call this 
ellipsoid in $\Redp$ the \emph{\jsellipsoid} of $f$ and denote it by 
$\sellbody{f}$. We call the $s$-marginal of $\sellbody{f}$ the 
\emph{\jsfunction} of $f$, and 
denote its density by
\[
 \selldense{f} = \text{ the density of } \smarg{\sellbody{f}}.
\] 

As a consequence of 
\eqref{eq:scaleliftf}, we note the \emph{scaling property} of $s$-ellipsoids: 
for any $s,\gamma>0$,
\begin{equation}\label{eq:scaleellips}
{\upthing{E}}
\text{ is the John } s\text{-ellipsoid  of } f
\text{ if and only if } 
{\left(I\oplus(\gamma^{1/s})\right)\upthing{E}}
\text{ is the John}\ s\text{-ellipsoid  of }\gamma f,
\end{equation}
or, equivalently, 
$
 \selldense{f}
$
is the \jsfunction{} of $f$
 if and only if
${\gamma  \cdot \selldense{f}}$
is the \jsfunction{} of $\gamma f.$
Similarly, for any affine map $\mathcal{A}: \Red \to \Red,$
$
 \selldense{f}
$
is the \jsfunction{} of $f$
 if and only if 
${\selldense{f}}\circ \mathcal{A}$
is the \jsfunction{} of $f \circ \mathcal{A}.$

\subsection{How the definitions described above implement the idea described in 
\ref{sec:motivation}}\label{sec:motivationagain}

We return to the case when $s$ is a positive 
integer. We first recall a classical definition.

We regard $\Re^{d+s}$ as the orthogonal sum $\Re^{d+s}=\Red\oplus\Re^s$, and 
denote by $\ball{s}$ the unit ball of $\Re^s$.
Let $\upthing{K}$ be a convex body in $\Re^{d+s}$. The \emph{Schwarz 
symmetrization} of $\upthing{K}$ about $\Red$ is the set
\[
 \upthing{K}^{\prime}=\bigcup \left\{r\ball{s}+x\st x\in P 
\!\left(\upthing{K}\right), 
 \vol{s}\left(r\ball{s}\right)= 
\vol{s}\left(\upthing{K}\cap\left(x+\Re^s\right)\right)\right\},
\]
where $P$ denotes the orthogonal projection from $\Re^{d+s}$ onto $\Red$, cf. 
\cite[Section~9.2.1.I]{BZ88}.
  As a well known consequence of the Brunn--Minkowski inequality, we have that 
$\upthing{K}^{\prime}$ is a convex body in $\Re^{d+s}.$ It is immediate from the 
definition that 
$\vol{d+s} \upthing{K}^{\prime} = 
\vol{d+s}\upthing{K}$, 
and more generally, the marginal on $\Red$ of the uniform measure on 
$\upthing{K}$ is identical to the marginal of the uniform measure on 
$\upthing{K}^{\prime}$.

The following claim follows from our definitions, we leave the proof to the 
reader.
\begin{prp}
Let $d,s>0$ be positive integers and let the function $f:\Red\rightarrow\Re$ be 
the density of the marginal on $\Red$ of the uniform measure on a convex body 
$\upthing{K}$ in $\Re^{d+s}$. Let $\upthing{K}^{\prime}$ denote the Schwarz 
symmetrization of $\upthing{K}$ about $\Red$, and let $\upthing{E}$ denote the 
John ellipsoid of $\upthing{K}^{\prime}$.
Then the marginal of the uniform measure on $\upthing{E}$ is the \jsfunction{} 
of $f$.
\end{prp}

\subsection{The height function of an ellipsoid, and formulation of our problem 
as functional optimization}\label{sec:height}
For any $(A \oplus \alpha, a) \in \ellips,$ we will say that $\alpha$ is the 
\emph{height} of the ellipsoid $\upthing{E}=(A \oplus \alpha)\ball{d+1} + a.$
We define the \emph{height function} of $\upthing{E}$ as
\[
\shf{\upthing{E}} (x) = 
 \begin{cases}
\alpha \sqrt{1 - \iprod{A^{-1} (x - a)}{A^{-1} (x - a)}},& 
\text{ if }x \in A \ball{d}+a\\
0,&\text{ otherwise}.
 \end{cases}
\]
Note that the height function of an ellipsoid is a proper log-concave function.
Clearly, the inclusion $\upthing{E} \subset \slift{f}$
holds if and only if
\begin{equation}
\label{eq:height_relation}
\shf{\upthing{E}} (x + a) \leq f^{1/s}(x + a) \text{ for 
all }
x \in A \ball{d}.
\end{equation}

As a closing note of the present introductory section, we rephrase our problem 
in a less geometric, more analytical language.

The classical John ellipsoid can be introduced as follows.
We consider the  class  of all nonsingular affine images (we may call them 
\emph{positions}) of the unit ball $\ball{d}$
contained in a given convex body $K$. The John ellipsoid is the (unique) largest 
volume element of this family. 

With the notion of height functions, it is easy to extend this approach to the 
setting of log-concave function. 
For any $s>0$, the  \jsfunction{} of a proper log-concave function $f$ on $\Red$ 
is the (unique) solution to the problem
\[
\max\limits_{h} \int_{\R^d} h^{s},
\]
where the maximum is taken over those positions 
\[
\left\{
h(x)=\alpha \cdot \shf{\ball{d+1}}(A^{-1}(x-a)),\ \text{where}\ a \in \Red,\ 
A \in \operatorname{GL}(d), \ \alpha > 0
\right\}
\]
of the height function 
\begin{equation}\label{eq:heightball}
\shf{\ball{d+1}}(x) = 
 \begin{cases}
\sqrt{1 - |x|^2},& 
\text{ if }x \in \ball{d}\\
0,&\text{ otherwise}
 \end{cases} 
\end{equation}
of the unit ball $\ball{d+1}$ which satisfy $\shf{\ball{d+1}}(A^{-1}(x-a))\leq 
f(x)$ for all $x\in\Red$. 

It follows from the polar decomposition theorem that we may restrict the set of 
positions to those where $A$ is a positive definite matrix.

\section{Some basic inequalities}
\label{sec:basicbounds}
\subsection{The \tpdfs-volume of ellipsoids}\label{sec:basicboundsellipsoid}
We denote the $s$-volume of the ball $\ball{d+1}$ of unit radius 
centered at the origin in $\Redp$ by $\volbs$, and compute it using 
spherical coordinates.
\begin{equation}\label{eq:kappasfirst}
 \volbs =\smeasure{\ball{d+1}}=
 \int\limits_{\ball{d}}\left(\sqrt{1 - |x|^2}\right)^{s} \di x =  
 \vol{d-1}S \int\limits_0^1 r^{d-1} (\sqrt{1 - r^2})^{s} \di r =  
\end{equation}
\begin{equation*}
 \frac{\vol{d-1} S}{2}
\int\limits_0^1 t^{(d-2)/2} (1 - t)^{s/2} \di t = 
 \frac{d\vol{d} \ball{d}}{2}
  \frac{\gammaf{ s/2 +1} \gammaf{ d/2}}{
  \gammaf{ s/2 + d/2 + 1}} = 
  \pi^{d/2} 
  \frac{\gammaf{ s/2 +1}}{ \gammaf{ s/2 + d/2 + 1}} , 
\end{equation*}
where $S=\bd{\ball{d}}$ denotes the unit sphere in $\Red$, and $\gammaf{\cdot}$ 
is Euler's Gamma function.

Note that
\begin{equation}\label{eq:kappaszero}
\lim_{s\to0^{+}} \volbs=\vol{d} \ball{d}.
\end{equation}
Thus, $\volbs,$ as a function of $s$ on  $[0, \infty)$
with $\volbs[0] = \vol{d} \ball{d},$  is a strictly decreasing continuous 
function on $[0, \infty)$.

By \eqref{eq:lintrafo}, the $s$-volume of a \dsymm{} ellipsoid can be expressed 
as
\begin{equation}\label{eq:sfunc_of_ellipsoida}
\smeasure{(A\oplus\alpha)\ball{d+1}+a} = \volbs \alpha^s \det A,
\end{equation}
 for any $(A\oplus\alpha, a)\in\ellips.$

\subsection{Bounds on \texorpdfstring{$f$}{f} based on local behaviour}

\begin{lemma}\label{lem:touching_cond_log_concave}
Let $\psi_1$ and $\psi_2$ be  convex functions on $\Red$ and $f_1= e^{-\psi_1}$ 
and $f_2 = e^{-\psi_2}.$ 
Let $f_2\leq f_1$ and $f_1(x_0) = f_2(x_0) >0$ at some point $x_0$ 
in 
the interior of the domain of $\psi_2.$ 
Assume that $\psi_2$ is differentiable at $x_0.$ Then $f_1$ and $f_2$ 
are differentiable at $x_0,$ 
$\nabla f_1(x_0)= \nabla f_2(x_0)$ 
and the following holds
\[
f_1(x) \leq f_2(x_0)e^{- \iprod{\nabla\psi_2(x_0)}{x-x_0}}
\]
for all $x \in \R^d.$
\end{lemma}
\begin{proof}
Since $f_2\leq f_1,$ the epigraph of $\psi_1$ contains the epigraph of 
$\psi_2$. 
Next, since $f_1(x_0) = f_2(x_0)$  and $f_2(x_0)$ is differentiable at $x_0$,
we conclude that both $f_1$ and $f_2$ have finite values and are continuous in a 
neighborhood of $x_0$
(see  Proposition~2.2.6 of \cite{clarke1990optimization}).

Using the Subdifferential Maximum Rule (see  Proposition~2.3.12 of 
\cite{clarke1990optimization}),
we see that $\psi_1$ is differentiable at $x_0$ and 
$\nabla \psi_1(x_0) = \nabla \psi_2(x_0)$, since the subdifferential
of $\psi_2$ at $x_0$ consists of the single vector $\nabla \psi_2(x_0)$. 

By the convexity of $\psi_1$, we have
\[
\psi_1(x) \geq \psi_1 (x_0) + \iprod{\nabla \psi_1(x_0)}{x - x_0} = \psi_2 
(x_0) + \iprod{\nabla \psi_2(x_0)}{x - x_0}
\]
for all $x \in \R^d$, and the result follows.
\end{proof}

\begin{cor}\label{cor:touchingballbound}
 Let $f$ be a log-concave function on $\Red$, and $s>0$. Assume that 
$\ball{d+1}\subseteq\slift{f}$ and $\upthing{u}\in\Redp\setminus\Red$ is a 
\emph{contact point} of $\ball{d+1}$ and $\slift{f}$, that is, 
$\upthing{u}\in\bd{\ball{d+1}}\cap\bd{\slift{f}}\setminus\Red$. Then
\begin{equation}\label{eq:touchingballbound}
 f(x)\leq w^s e^{- \frac{s}{w^2}\iprod{u}{x-u}}
 \;\;\text{ for all } x\in\Red,
\end{equation}
where $u$ is the orthogonal projection of $\upthing{u}$ onto $\Red$ and 
$w=\sqrt{1-|u|^2}$.
\end{cor}
Note that since $u\notin\Red$, we have $w>0$.

\begin{proof}[Proof of Corollary~\ref{cor:touchingballbound}]
Applying \Href{Lemma}{lem:touching_cond_log_concave} to the functions 
$f_1=f^{1/s}$ and $f_2=\shf{\ball{d+1}}$ at $x_0=u$, we obtain
\[
 f^{1/s}(x)\leq w 
 e^{- \iprod{\nabla \left[-\log \shf{\ball{d+1}}\right](u)}{x-u}}
 \;\;\text{ for all } x\in\Red.
\]
Since for any $y\in\inter{\ball{d}}$, we have
\[
\nabla \left[-\log \shf{\ball{d+1}}\right](y)= 
-\frac{1}{2}\nabla \left[\log (1- |y|^2) \right]
= \frac{y}{1 - |y|^2},
\]
inequality \eqref{eq:touchingballbound} follows.
\end{proof}

\subsection{Compactness}

We show that ellipsoids of large $s$-volume contained 
in $\slift{f}$ are contained in a bounded region of $\Redp$. 
We phrase the next lemma in a general functional language, but we will apply it 
mostly for $g=\shf{\ball{2}}$, the height function of the ball $\ball{2}$ (see 
\eqref{eq:heightball}).
 
\begin{lemma}[Compactness]\label{lem:boundedness}
For any  proper log-concave function $f:\Red\to[0,\infty)$ and any $\delta>0$, 
there exist  $\vartheta,\rho,\rho_1, \rho_2>0$ with the following property. If  
for
a  proper even log-concave function 
$g:\R \to [0,\infty)$ with $g(0)=1$ and 
$(A\oplus\alpha,a)\in\ellips$, the function
$\tilde{g} : \R^d \to [0,\infty)$ given by 
\[\tilde{g}(x) = \alpha g\left(\left|A^{-1}(x-a)\right|\right)\]
satisfies $\tilde{g} \leq f$  and 
$ \int_{\Red} \tilde{g}  \geq\delta$, then the following inequalities hold.
\begin{equation}\label{eq:alpha_a_universal_bound}
 \vartheta\leq\alpha \leq \norm{f}
\quad \text{and} \quad 
 |a|  \leq \rho,
\end{equation}
and
\begin{equation}
\label{eq:comparison_operator}
\rho_1 \frac{\left(\int_{\R} g\right)^{d-1}}{ \int_{\Red} {g}(|x|) \di x } \cdot 
I \prec {A} \prec \frac{\rho_2}{ \int_{\R} g } \cdot I.
\end{equation}
\end{lemma}
\begin{proof}
Obviously,  $\alpha \leq \norm{f}.$
To bound $\alpha$ from below, we fix $\vartheta$ with
$\alpha \leq \vartheta.$ Then 
$\tilde{g} \leq \vartheta$, and thus,
\[
 \int_{\Red} \tilde{g} \leq
\int_{\Red} \min\{f(x), \vartheta\} \di x.
\]
Since $f$ is a non-negative function of finite integral, the last expression is 
less than $\delta$ if $\vartheta$ is sufficiently small. Thus, the leftmost 
inequality in  \eqref{eq:alpha_a_universal_bound} holds. 
Since $\tilde{g}(a)  = \alpha,$ we conclude that
$a \in [f \geq \vartheta]$ completing the proof of    
\eqref{eq:alpha_a_universal_bound}.

We proceed with inequality~\eqref{eq:comparison_operator}.
Let $\ell$ be the line passing through $a$ in the direction of an eigenvector of 
$A$ corresponding to the eigenvalue $\norm{A}.$ 
We have
\[
 \int_{\ell} f \geq \int_{\ell}\tilde{g}={\alpha}{\norm{A}} \int_{\R} g.
\]
On the other hand, there exists a positive constant $C_f $ such that the 
integral of the proper log-concave function $f$ over any line is
at most $C_f.$ It follows, for example, from the existence of constants
 $\Theta, \nu > 0$ depending only on $f$
such that 
\begin{equation*}
	f(x) \leq \Theta e^{-\nu |x|} 
\end{equation*} 
{ for all } $x\in\Red,$
see \cite[Lemma~2.2.1] 
{brazitikos2014geometry}. Thus, the rightmost relation in 
\eqref{eq:comparison_operator} holds with $\rho_2 = 2\frac{C_f}{\vartheta}$.

By the assumption, we have
\[
  \delta \leq \int_{\R^d} \tilde{g}  = \alpha  \det A  \cdot
 \int_{\R^d} g(|x|) \di x.
\]
Let $\beta$ be the smallest eigenvalue of $A.$ By the previous inequality and 
since $\alpha \in [\vartheta, \norm{f}],$ we have
\[
0 < \frac{\delta}{\norm{f}} \frac{1}{ \int_{\R^d} g(|x|) \di x} \leq 
\det A \leq \beta \norm{A}^{d-1}.
\]
By the rightmost relation in \eqref{eq:comparison_operator}, the existence of 
$\rho_1$ follows.
\end{proof}

\section{Interpolation between ellipsoids}\label{sec:interpolation}

In this section, we show that if two ellipsoids are contained in $\slift{f}$, 
then we can define a third ellipsoid that is also contained in $\slift{f}$, and 
we give a lower bound on its $s$-volume. The latter is a Brunn--Minkowski type 
inequality for the $s$-volume of ellipsoids.

After preliminaries, we present the main results of this section in 
\Href{Subsection}{subsec:keylemmas}, which is followed by immediate 
applications, 
one of which is the proof of the existence and uniqueness of the John 
$s$-ellipsoid (\Href{Theorem}{thm:johnunicity}).

\subsection{Operations on functions: Asplund sum, 
epi-product}\label{sec:asplund}
Following Section~9.5 of \cite{schneider2014convex}, we define the 
\emph{Asplund sum} (or 
\emph{sup-convolution}) of two log-concave 
functions $f_1$ and $f_2$ on $\Red$ by
\[
(\loginfconv{f_1}{f_2})(x) = \sup\limits_{x_1 + x_2 = x} f_1(x_1) f_2(x_2),
\]
and the \emph{epi-product} of a log-concave function $f$ on $\Red$ with a 
scalar $\lambda>0$ by
\[
(\lambda \ast f) (x) = f\!\left(\frac{x}{\lambda}\right)^\lambda.
\]

Clearly, 
\[
\norm{f} = \norm{f_1} \norm{f_2}, \text{ where } f =\loginfconv{f_1}{f_2}.
\]
It is easy to see that for any proper log-concave function $f$ and $\lambda \in 
[0,1],$
we have
\begin{equation}
\label{eq:loginvconv_same_f}
\loginfconv{(\lambda \ast f)}{((1 - \lambda) \ast f)} = f.
\end{equation}

As motivation for the definitions above, we describe a geometric interpretation 
of the Asplund sum: analogy with the Minkowski sum of convex bodies in $\Red.$
Let $\psi_1, \psi_2: \Red \to \R$  be two convex functions
that attain their  minimums. Then the Asplund sum of  $f_1 = 
e^{-\psi_1}$ 
and $f_2 = e^{- \psi_2}$ equals 
\[
\loginfconv{f_1}{f_2} = e^{-\psi},
\]
where $\psi$ is the function defined by taking the Minkowski sum of the 
epigraphs, that is,
\[
\operatorname{epi} \psi = {\operatorname{epi} \psi_1 + 
\operatorname{epi}\psi_2}.
\]

\subsection{A non-linear combination of two ellipsoids}\label{subsec:keylemmas}

The following two lemmas are our key tools. They allow us to interpolate 
between two ellipsoids.

\begin{lemma}[Containment of the interpolated 
ellipsoid]\label{lem:inclusion_ellipsoids}
Fix $s_1, s_2,\beta_1,\beta_2 > 0$ with
$\beta_1  + \beta_2 = 1$.   Let $f_1$ and $f_2$ be two proper log-concave 
functions on $\Red$, and 
$\upthing{E_1}, \upthing{E_2}$ be two \dsymm{} ellipsoids represented 
by $(A_1 \oplus \alpha_1, a_1) \in \ellips$ and $(A_2 \oplus \alpha_2, a_2) \in 
\ellips,$ respectively, such that 
\begin{equation}
\label{eq:contaiment_ass_inclusion}
\upthing{E}_1 \subset \slift[s_1]{f_1} \quad \text{   and   } \quad
\upthing{E_2} \subset \slift[s_2]{f_2}.
\end{equation}
Define
\[
f = \loginfconv{(\beta_1 \ast f_1)}{(\beta_2 \ast f_2)}
\quad
\text{and}
\quad
s =  \beta_1 s_1 +  \beta_2 s_2.
\]
Set
\[
(A \oplus \alpha, a) = 
\left(\left(\beta_1 A_1 + \beta_2 A_2\right) \oplus (\alpha_1^{\beta_1 
s_1}\alpha_2^{\beta_2 s_2})^{1/s}, \beta_1 a_1 + \beta_2 a_2
\right) \quad \text{and} 
\quad
\upthing{E} = (A \oplus \alpha) \ball{d+1} + a.
\]
Then,
\begin{equation}\label{eq:inclusion_ellipsoids}
\upthing{E}  \subset \slift{f}.
\end{equation}
\end{lemma}

\begin{proof}
Fix $x \in A \ball{d}$ and define
\[
x_1 =  A_1 A^{-1} x, \quad x_2 =  A_2 A^{-1} x. 
\]
Clearly, $x_1 \in A_1 \ball{d}$ and $x_2 \in A_2 \ball{d}.$
Thus, by  \eqref{eq:contaiment_ass_inclusion} and \eqref{eq:height_relation}, 
we have  
\begin{equation}\label{eq:mink_selippses_heights}
f_1^{1/s_1} (x_1 + a_1) \ge  \shf{\upthing{E}_1}(x_1 + a_1) 
\quad \text{and} \quad 
f_2^{1/s_2} (x_2 + a_2) \ge  \shf{\upthing{E}_2}(x_2 + a_2).
\end{equation}
By our definitions, we have that 
$ \beta_1 (x_1 + a_1) + \beta_2(x_2 + a_2)  = x + a.$
Therefore,  by the definition of the Asplund sum,  we  have that
\[
f(x + a)    \ge f_1^{\beta_1}( x_1 + a_1) f_2^{\beta_2}( x_2 + a_2),
 \]
which by \eqref{eq:mink_selippses_heights}, yields
\[
{f(x + a)}  \ge 
  \left(\shf{\upthing{E}_1}(x_1 + a_1))^{ \beta_1 s_1}( \shf{\upthing{E}_2}(x_2 
+ 
a_2)\right)^{\beta_2 s_2}.
\]
By the definition of the height function, and since 
$A^{-1}x=A_1^{-1}x_1=A_2^{-1}x_2$, we have
\begin{gather*}
 \left(\shf{\upthing{E}_1}(x_1 + a_1))^{ \beta_1 s_1}( \shf{\upthing{E}_2}(x_2 
+ 
a_2)\right)^{\beta_2 s_2} = \\
 \left(\alpha_1 \sqrt{1 - \iprod{A_1^{-1} x_1}{ A_1^{-1} x_1}} \right)^{\beta_1 
s_1} \left(\alpha_2 \sqrt{1 - \iprod{A_2^{-1} x_2}{ A_2^{-1} x_2}} 
\right)^{\beta_2 s_2} = \\
  \alpha_1^{\beta_1 s_1} \alpha_2^{\beta_2 s_2}  \left(\sqrt{1 - \iprod{A^{-1} 
x}{A^{-1} x}}\right)^{\beta_1 s_1 + \beta_2 s_2} =  \left( \alpha \sqrt{1 - 
\iprod{A^{-1} x}{A^{-1} x}} \right)^s=
 \left(\shf{\upthing{E}}(x + a)\right)^s.
 \end{gather*}
Combining this with the previous inequality, we obtain inequality 
\eqref{eq:height_relation}.
This completes the proof.
\end{proof}
\begin{lemma}[Volume  of the interpolated 
ellipsoid]\label{lem:minkowski_for_s_ellipsoids} 
Under the conditions of \Href{Lemma}{lem:inclusion_ellipsoids} with $s=s_1 = 
s_2$,
the following inequality holds.
\begin{equation}\label{eq:minkowski_for_s_ellipsoids}
\smeasure{\upthing{E}} \ge \left(\smeasure{\upthing{E}_1}\right)^{\beta_1}
\left(\smeasure{\upthing{E}_2}\right)^{\beta_2},
\end{equation}
with equality if and only if $A_1 = A_2.$ 
\end{lemma} 
\begin{proof}
We set $s=s_1=s_2$, and observe that
by \eqref{eq:sfunc_of_ellipsoida}, inequality
\eqref{eq:minkowski_for_s_ellipsoids} is equivalent to
\[
\volbs (\alpha_1^{\beta_1} \alpha_2^{\beta_2})^s \cdot \det{\left(\beta_1 A_1 + 
\beta_2 A_2 \right)} \ge
\volbs (\alpha_1^{\beta_1} \alpha_2^{\beta_2})^s \cdot \left(\det A_1 
\right)^{\beta_1} \left(\det A_2\right)^{\beta_2},
\]
which holds if and only if
\[
 \det{\left(\beta_1 A_1 + \beta_2 A_2 \right)} \ge
\left(\det A_1 \right)^{\beta_1} \left(\det A_2\right)^{\beta_2}.
\]
Finally, \eqref{eq:minkowski_for_s_ellipsoids} and its equality condition 
follow from Minkowski's determinant inequality 
\eqref{eq:minkowski_det_multipl_ineq} and the equality condition 
therein.
\end{proof}

\subsection{Uniqueness of the John \texorpdfstring{$s$}{s}-ellipsoids}
We start with a simple but useful observation.

\begin{lemma}[Interpolation between translated ellipsoids]
\label{lem:interpolation_translates}
Let $f$ be a proper  log-concave function on $\Red$, and $s>0$.
Assume that the two \dsymm{} ellipsoids $\upthing{E}_1$ and $\upthing{E}_2$ 
contained in $\slift{f}$ are translates of each other by a vector in $\Red$. 
More specifically, assume that they are represented by 
$(A \oplus \alpha, {a_1})$ and 
$(A \oplus \alpha, {a_2})$ with $a_1 = - a_2 = \delta A e_1,$
where $e_1$ is the first standard basis vector in $\Red$, and $\delta>0$.
Then the origin centered ellipsoid 
\[
\upthing{E}_0 =  (A \oplus \alpha)  \upthing{M} \ball{d+1}, \quad 
\text{where} \;\; 
\upthing{M} = \diag(1+\delta, 1, \dots, 1),
\]
is contained in $\slift{f}$.
\end{lemma}
\begin{proof}
Since all super-level sets of $f^{1/s}$ are convex sets in $\Red$, it is easy 
to see that for any convex set $\upthing{H}$ in $\Redp$ and vector $v\in\Red$, 
if 
$\upthing{H}\subseteq\slift{f}$ and $\upthing{H}+v\subseteq\slift{f}$, then 
$ \conv{\upthing{H}\cup(\upthing{H}+v)}=
\upthing{H}+[0,v]\subseteq\slift{f}$.

Thus, $\slift{f}$ contains the ``sausage-like'' body
 \[
 \upthing{W} = 
 \conv{\upthing{E}_1\cup\upthing{E}_2}=
 (A \oplus \alpha) \left(\ball{d+1} + 
 [ A^{-1} a_2, A^{-1} a_1 ] \right)= (A \oplus \alpha) \left(\ball{d+1} 
+ 
 [ -\delta e_1, \delta e_1] \right).
\]
On the other hand, clearly, $\upthing{E}_0 \subseteq \upthing{W}$ completing 
the proof of \Href{Lemma}{lem:interpolation_translates}.
\end{proof}

As an application of Lemmas~\ref{lem:inclusion_ellipsoids} and 
\ref{lem:minkowski_for_s_ellipsoids}, we show that  in the set of  \dsymm\ 
ellipsoids  in $\slift{f}$ with a fixed height, a  largest $s$-volume \dsymm\  
ellipsoid is unique.
 
\begin{lemma}[Uniqueness for a fixed height]
\label{lem:uniquness_s_ellips_fixed_height}
Let $f$ be a proper  log-concave function on $\Red$, and $s>0$. Then, among all 
\dsymm{} ellipsoids of height $\alpha, 0 < \alpha < \norm{f}^{1/s},$ in 
$\slift{f}$, there is a unique one of maximal $s$-volume. Additionally, if 
there 
is a  \dsymm{} ellipsoid in $\slift{f}$ of height $\norm{f}^{1/s},$ then among 
all \dsymm{} ellipsoids of height $\alpha=\norm{f}^{1/s}$ in $\slift{f}$, there 
is a unique one of maximal $s$-volume.\end{lemma}
\begin{proof}
Clearly, the maximum $s$-volume among \dsymm{} ellipsoids of height $\alpha$ 
contained in $\slift{f}$ is positive. By \Href{Lemma}{lem:boundedness} applied 
with $g= \shf{\ball{2}}^s$, where $\shf{\ball{2}}$ is the height function of 
$\ball{2}$ (see 
\eqref{eq:heightball}), identity \eqref{eq:sfunc_of_ellipsoida} and a standard 
compactness argument, this maximum is attained.

We show that such an ellipsoid is unique.
Assume that 
 $\upthing{E}_1 \subset \slift{f}$ and 
$\upthing{E}_2  \subset \slift{f}$, represented by
 $(A_1 \oplus \alpha, a_1) \in \ellips$ and $(A_2 \oplus \alpha, a_2) \in 
\ellips$,
are two $d$-symmetric ellipsoids of the maximal $s$-volume.

Define a new $d$-symmetric ellipsoid $\upthing{E}$ represented by
\[
\left(\frac{A_1+A_2}{2} \oplus \alpha, \frac{a_1+a_2}{2} \right) \in \ellips.
\]

Applying \eqref{eq:loginvconv_same_f} with $\lambda=1/2$ and 
\Href{Lemma}{lem:inclusion_ellipsoids}, we have $\upthing{E} \subset \slift{f}$.
Next, by the choice of the ellipsoids, we have that 
\[
\smeasure{\upthing{E}} \leq 
\smeasure{\upthing{E}_1} =
\sqrt{\smeasure{\upthing{E}_1}
\smeasure{\upthing{E}_2}} =\smeasure{\upthing{E}_2} .
\]
By Lemma 4.2,  we have that $\smeasure{\upthing{E}} \geq
\sqrt{\smeasure{\upthing{E}_1}
\smeasure{\upthing{E}_2}}$, therefore equality holds.
Thus, by the equality condition in \Href{Lemma}{lem:minkowski_for_s_ellipsoids}, 
 we 
conclude that 
$A_1 = A_2$.

To complete the proof, we need to show that $a_1 = a_2$.
Assume the contrary: $a_1 \neq a_2$.
By translating the origin and rotating the space $\Red$, we may assume that 
$a_1  = - 
a_2 \neq 0$ and that $A_1^{-1} a_1 =\delta e_1$ for some $\delta>0$.

{By \Href{Lemma}{lem:interpolation_translates}, the ellipsoid 
$
\upthing{E}_0 =  (A_1 \oplus \alpha)  \upthing{M} \ball{d+1}
$
is contained in $\slift{f}$, where $\upthing{M} = \diag(1+\delta, 1, \dots, 1)$.
However, $\smeasure{\upthing{E}_0} > \smeasure{\upthing{E}} = \smeasure{\upthing{E}}_1$, which
contradicts the choice of  $\upthing{E}_1$ and $\upthing{E}_2$, completing 
the proof of \Href{Lemma}{lem:uniquness_s_ellips_fixed_height}.
}
\end{proof}

\begin{thm}[Existence and uniqueness of the \jsellipsoid]\label{thm:johnunicity}
Let $s > 0$ and $f$ be a proper log-concave function on $\Red.$
Then, there exists a unique \jsellipsoid{} of $f$. 
\end{thm}
\begin{proof}[Proof of \Href{Theorem}{thm:johnunicity}]
As in the proof of \Href{Lemma}{lem:uniquness_s_ellips_fixed_height}, the 
existence of an $s$-ellipsoid follows from 
\Href{Lemma}{lem:boundedness} applied 
with $g=\shf{\ball{2}}^s$, where $\shf{\ball{2}}$ is the height function of 
$\ball{2}$ (see 
\eqref{eq:heightball}),
identity \eqref{eq:sfunc_of_ellipsoida} and a standard compactness argument. 


Assume that 
$\upthing{E}_1 \subset \slift{f}$ and 
$\upthing{E}_2  \subset \slift{f}$ are
two $d$-symmetric ellipsoids of maximal $s$-volume,
represented by
$(A_1 \oplus \alpha_1, a_1) \in \ellips$ and 
$(A_2 \oplus \alpha_2, a_2) \in \ellips$, respectively. 
We define a new $d$-symmetric ellipsoid $\upthing{E}$ represented by
\[
\left(\frac{A_1+A_2}{2} \oplus \sqrt{\alpha_1 \alpha_2}, \frac{a_1+a_2}{2} 
\right) \in \ellips.
\]

Applying \eqref{eq:loginvconv_same_f} with $\lambda=1/2$ and 
\Href{Lemma}{lem:inclusion_ellipsoids}, we have $\upthing{E} \subset \slift{f}$.
Next, by the choice of the ellipsoids, we also have
\[
\smeasure{\upthing{E}} \leq 
\smeasure{\upthing{E}_1} =
\sqrt{\smeasure{\upthing{E}_1}
\smeasure{\upthing{E}_2}} =\smeasure{\upthing{E}_2},
\]
which, combined with 
\Href{Lemma}{lem:minkowski_for_s_ellipsoids}, yields
$
\smeasure{\upthing{E}} = 
\smeasure{\upthing{E}_1} =
\smeasure{\upthing{E}_2}
$
and $A_1 = A_2$.
This implies that $\alpha_1 = \alpha_2,$ since  the $s$-volume of 
$\upthing{E}_1$ and $\upthing{E}_2$ are equal.
Therefore, by \Href{Lemma}{lem:uniquness_s_ellips_fixed_height}, the ellipsoids 
$\upthing{E}_1$ and $\upthing{E}_2$ coincide, completing the proof of 
\Href{Theorem}{thm:johnunicity}.
\end{proof}

\subsection{Bound on the height}
Recall from \Href{Section}{sec:sellipsoid} that $\selldense{f}$ denotes the 
density of the \jsfunction{} of $f$, that is, the density of the 
$s$-marginal of the \jsellipsoid{} of $f$.
The following result is an extension of the analogous result on the ``height'' 
of \bernardoell{}  
\cite[Theorem 1.1]{alonso2018john}  to the John $s$-ellipsoid with a 
similar proof.
\begin{lemma}\label{lem:height}
Let $f$ be a proper  log-concave function on $\Red$ and $s > 0$. 
Then, 
\begin{equation}\label{eq:height}
 \norm{\selldense{f}} \geq e^{-d} \norm{f}. 
\end{equation}
\end{lemma}

We note that if the \jsellipsoid{} of $f$ is represented by $\left(A_s 
\oplus \alpha_s, a_s \right)$ (that is, its height is $\alpha_s$), then 
$\norm{\selldense{f}}=\alpha_s^{s}$.

\begin{proof}[Proof of Lemma~\ref{lem:height}]
We define a function $\Psi: (0, \norm{f}^{1/s}) \to \R^+$ as follows.
By \Href{Lemma}{lem:uniquness_s_ellips_fixed_height}, for any $\alpha\in (0, 
\norm{f}^{1/s})$, there is a unique
\dsymm{} ellipsoid of maximal $s$-volume among \dsymm{} ellipsoids
of height $\alpha$ in $\slift{f}$. Let this ellipsoid be represented by 
$\left(A_{\alpha} \oplus \alpha, a_\alpha\right)\in\ellips$. We set 
$\Psi(\alpha) = \det 
A_{\alpha}$.
\begin{claim}\label{claim:psilogbr}
For any $\alpha_1, \alpha_2 \in (0, \norm{f}^{1/s})$  and $\lambda \in [0,1],$ 
we have
\begin{equation}
\label{eq:ellips-d_quasi_log-concavity}
\Psi \!\left(\alpha_1^\lambda \alpha_2^{1-\lambda} \right)^{1/d} \ge \lambda 
\Psi(\alpha_1)^{1/d} + (1 - \lambda) \Psi(\alpha_2)^{1/d}.
\end{equation}
\end{claim}
\begin{proof}[Proof of Claim~\ref{claim:psilogbr}]
Let $\left(A_1 \oplus \alpha_1, a_1 \right)$ and  
$\left(A_2 \oplus \alpha_2, a_2 \right)$ represent 
the \dsymm{} ellipsoids of maximum $s$-volume contained in $\slift{f}$ with the 
corresponding heights.
By \Href{Lemma}{lem:inclusion_ellipsoids} and \eqref{eq:sfunc_of_ellipsoida}, 
we have that
\[
\Psi\! \left(\alpha_1^\lambda \alpha_2^{1-\lambda}\right) \geq \det 
\left(\lambda A_1 + (1-\lambda)A_2 \right).
\]
Now, \eqref{eq:ellips-d_quasi_log-concavity} follows immediately from
Minkowski's determinant inequality \eqref{eq:minkowski_det_ineq}.
\end{proof}
Set $\Phi(t) = \Psi\! \left(e^{t} \right)^{1/d}$ for all $t \in 
\left(-\infty,\frac{\log \norm{f}}{s}\right).$
Inequality \eqref{eq:ellips-d_quasi_log-concavity} implies that $\Phi$ is a 
concave function on its domain.

Let $\alpha_0$ be the height of the \jsellipsoid{} of $f.$ Then, by 
\eqref{eq:sfunc_of_ellipsoida}, for any 
$\alpha$ in the domain of $\Psi,$ we have that 
\[
\Psi(
\alpha) \alpha^s \leq \Psi(\alpha_0) \alpha_0^s.
\] 
Setting $t_0 =  \log \alpha_0$ and taking root of order $d,$ we obtain
\[
\Phi(t) \leq \Phi(t_0) e^{\frac{s}{d}(t_0 - t)}
\]
for any $t$ in the domain of $\Phi.$ The expression on the right-hand side is a 
convex function of $t$, while $\Phi$ is a concave 
function. Since these functions take the same value at $t = t_0,$ we conclude 
that the graph of $\Phi$ lies below the tangent line to graph of $\Phi(t_0) 
e^{\frac{s}{d}(t_0 - t)}$ at point $t_0.$ That is, 
\[
\Phi(t) \leq \Phi(t_0) \left(1 - \frac{s}{d}(t - t_0)\right).
\]
Passing to the limit as $t \to  \frac{\log \norm{f}}{s}$  and since the values 
of $\Phi$ are  positive, we get
\[
0 \leq 1 - \frac{\log \norm{f}}{d} + \frac{s}{d} t_0.
\]
Or, equivalently, $t_0 \geq - \frac{d}{s} + \frac{\log \norm{f}}{s}.$
Therefore, $\alpha_0 \geq  e^{-d/s} \norm{f}^{1/s}$ and ${\norm{\selldense{f}}} 
= \alpha_0^s \geq e^{-d} \norm{f}.$
This completes the proof of \Href{Lemma}{lem:height}.
\end{proof}

\section{John's condition --- Proof of Theorem~\ref{thm:johncondbasic}}
\label{sec:johncond}

\Href{Theorem}{thm:johncondbasic} is an immediate consequence of the following 
theorem whose proof is the topic of this section.

\begin{thm}\label{thm:johncond}
Let $\upthing{K}$ be a closed \dsymm{} set in $\Redp$, and let $s>0$. Assume 
that $\ball{d+1}\subseteq\upthing{K}$. Then the 
following hold.

\begin{enumerate}
 \item\label{part:localToIsotropic} 
Assume that $\ball{d+1}$
is a \emph{locally} maximal $s$-volume ellipsoid contained in $\upthing{K}$, 
that is, in some neighborhood of $\ball{d+1}$, no ellipsoid contained in 
$\upthing{K}$ is of larger $s$-volume. 

\noindent Then there are contact points 
$\upthing{u}_1,\ldots,\upthing{u}_k\in\bd{\ball{d+1}}\cap\bd{\upthing{K}}$ 
and positive weights $c_1,\ldots,c_k$ such that 
\begin{equation}\label{eq:johncond}
 \sum_{i=1}^k c_i \upthing{u}_i\otimes \upthing{u}_i =\upthing{S}
 \;\;\;\mbox{ and }\;\;\;\;
 \sum_{i=1}^k c_i u_i =0,
\end{equation}
where $u_i$ is the orthogonal projection of $\upthing{u}_i$ onto $\Red$ and 
$\upthing{S}=\diag(1,\ldots,1,s)=I\oplus s$. Moreover, such contact points and 
positive
weights exist for some $k$ with $d+1\leq k\leq \frac{(d+1)(d+2)}{2}+d+1$.

 \item\label{part:IsotropicToGlobal}
Assume that $\upthing{K}=\slift{f}$ for a proper log-concave function $f$, and 
that 
there are contact points and positive weights satisfying \eqref{eq:johncond}. 

\noindent Then $\ball{d+1}$ is the unique ellipsoid of (\emph{globally}) 
maximum 
$s$-volume among \dsymm{} ellipsoids contained in $\upthing{K}$.
\end{enumerate}
\end{thm}

We equip $\MM$ (for the definition, see \eqref{eq:Mdef}) with an inner product 
(that comes from the Frobenius product on 
the space of matrices and the standard inner product on $\Red$) defined by
\[ 
\iprod{(\upthing{A},a)}{(\upthing{B},b)}=
 \tr{{\upthing{A}}\,{\upthing{B}}}+\iprod{a}{b}.
\]
Thus, we may use the topology of $\MM$ on the set $\ellips$ of ellipsoids in 
$\Redp$.

Denote the set of contact points by 
$C=\bd{\ball{d+1}}\cap\bd{\upthing{K}}$, and consider
\[
 \widehat{C}=\left\{(\upthing{u}\otimes\upthing{u}, u)\st 
\upthing{u}\in C\right\}\subset \MM,
\]
where $u$ denotes the orthogonal projection of $\upthing{u}$ onto $\Red$.

The proof of Part~\eqref{part:localToIsotropic} of \Href{Theorem}{thm:johncond} 
is an adaptation of the argument given in \cite{ball1997elementary} and 
\cite{GruberBook} (see also \cite{GLMP04, GPT01, BR02,L79} and 
\cite[Theorem~14.5]{TJ89}) to the $s$-volume.
The idea is that, if there are no contact points and positive weights 
satisfying 
\eqref{eq:johncond}, then there is a path, namely a \emph{line segment} in 
the space $\ellips$ of ellipsoids starting from $\ball{d+1}$ such that the 
$s$-volume 
increases along the path and the path stays in the family of ellipsoids 
contained in 
$\upthing{K}$.

{
Part~\eqref{part:IsotropicToGlobal} on the other hand, needs a finer argument.
The idea is that, if $\ball{d+1}$ is not the global maximizer of the 
$s$-volume, then we will find a path in $\ellips$ starting from $\ball{d+1}$ 
such that the $s$-volume increases along 
the path, and the path stays in the family of ellipsoids contained in 
$\upthing{K}$. 
The difficulty is that $\slift{f}$ is not necessarily convex. Thus, this path 
\emph{is not a line segment}. We will, however, be able to differentiate the 
$s$-volume along 
this path, and by doing so, we will show that $(\upthing{S},0)$ is separated by 
a hyperplane from the points $\widehat{C}$ in $\MM$, which in turn will yield 
that there are no contact points and positive weights satisfying 
\eqref{eq:johncond}.
}

First, as a standard observation, we state the relationship between 
\eqref{eq:johncond} and separation by a hyperplane of the point 
$(\upthing{S},0)$ from the set $\widehat{C}$ in the space $\MM$.

\begin{claim}\label{claim:Hseparates}
The following assertions are equivalent.
\begin{enumerate}
 \item\label{item:eqjohncond} 
 There are contact points and positive weights satisfying \eqref{eq:johncond}.
 \item\label{item:eqjohncondmodified} 
 There are contact points and positive weights satisfying a modified version of 
\eqref{eq:johncond}, where in the second equation $u_i$ is replaced by 
$\upthing{u}_i$. 
 \item\label{item:sinpos}
 $(\upthing{S},0)\in\pos(\widehat{C})$.
 \item\label{item:sinconv} 
 $\frac{1}{d+s}(\upthing{S},0)\in\conv{\widehat{C}}$.
 \item\label{item:separation} 
There is no $(\upthing{H},h)\in\MM$ with
\begin{equation}\label{eq:Hseparates}
 \iprod{(\upthing{H},h)}{(\upthing{S},0)}>0, \text{ and }
 \iprod{(\upthing{H},h)}{(\upthing{u}\otimes\upthing{u},u)}<0 
 \text{ for all }\upthing{u}\in C.
\end{equation}
\item\label{item:separationweaker}
There is no $(\upthing{H},h)\in\MM$ with
\begin{equation}\label{eq:Hseparatesweak}
 \iprod{(\upthing{H},h)}{(\upthing{S},0)}>0, \text{ and }
 \iprod{(\upthing{H},h)}{(\upthing{u}\otimes\upthing{u},u)}\leq0 
 \text{ for all }\upthing{u}\in C.
\end{equation}
\end{enumerate}
\end{claim}

\begin{proof}
We leave it to the reader to verify the equivalence of \eqref{item:eqjohncond} 
and \eqref{item:eqjohncondmodified} and \eqref{item:sinpos}, as well as that of
\eqref{item:separation} and \eqref{item:separationweaker}.

To see that \eqref{item:sinpos} is equivalent to \eqref{item:sinconv}, we
take trace in \eqref{eq:johncond} and notice that 
$\tr{\upthing{u}\otimes\upthing{u}}=\tr{\frac{1}{d+s}\upthing{S}}=1$,
which shows that $\sum_{i=1}^k c_i=d+s$.

Finally, observe that the convex cone $\pos(\widehat{C})$ in $\MM$ does not 
contain 
the point $\left(\upthing{S},0\right)\in\MM$ if and only if it is separated 
from this point by a hyperplane through the origin. This is what 
\eqref{eq:Hseparates} expresses, showing that \eqref{item:sinpos} is equivalent 
to \eqref{item:separation}, and hence, completing the proof of 
\Href{Claim}{claim:Hseparates}.
\end{proof}

\begin{claim}\label{claim:caratheodory}
If contact points and positive weights satisfying \eqref{eq:johncond} exist for 
some 
$k$, then 
they exist for some $d+1\leq k\leq \frac{(d+1)(d+2)}{2}+d+1$.
\end{claim}
\begin{proof}
Since $\upthing{u}\otimes\upthing{u}$ is of rank 1, the lower bound on $k$ is 
obvious. The upper bound follows from \eqref{item:sinconv} in 
\Href{Claim}{claim:Hseparates} and Carath\'eodory's theorem applied in the 
vector 
space $\MM$.
\end{proof}

Next, we show that if $(\upthing{S},0)$ and $\widehat{C}$ are separated by a 
hyperplane in $\MM$, then the normal vector of that hyperplane can be chosen 
to be of a special form.
\begin{claim}\label{claim:HisinM}
There is $(\upthing{H},h)\in\MM$ satisfying \eqref{eq:Hseparates} if and only 
if there is 
$(\upthing{H}_0,h) \in \MM$ satisfying \eqref{eq:Hseparates}, where 
$\upthing{H}_0=H_0\oplus\gamma$ for some $H_0\in\Re^{d\times d}$.
\end{claim}

\begin{proof}
For any $\upthing{u}\in\Redp$, let $\upthing{u}^{\prime}$ denote the reflection 
of $\upthing{u}$ about $\Red$, that is, $\upthing{u}^{\prime}$ differs from 
$\upthing{u}$ only in the last coordinate, which is the opposite of the last 
coordinate of $\upthing{u}$. Since both $\upthing{K}$ and $\ball{d+1}$ are 
symmetric about $\Red$, we conclude that, if $\upthing{u}$ is in $C$, then so 
is 
$\upthing{u}^{\prime}$. 

Let $\upthing{H}_0$ denote the matrix obtained 
from $\upthing{H}$ by setting the first $d$ entries of the last row to 
zero, and the first $d$ entries of the last column to zero. Thus, 
$\upthing{H}_0$ is of the required form. We show that 
$(\upthing{H}_0,h)$ satisfies \eqref{eq:Hseparates}. Clearly, 
$
\iprod{\left(\upthing{H},h\right)}{\left(\upthing{S},0\right)}=
\iprod{\left(\upthing{H}_0,h\right)}{\left(\upthing{S},0\right)}$, and 
thus, the first inequality in \eqref{eq:Hseparates} holds. 

For the other 
inequality in \eqref{eq:Hseparates}, consider an arbitrary vector 
$\upthing{u}\in C$. Then the inequalities
$0>\iprod{(\upthing{H},h)}{(\upthing{u}\otimes\upthing{u},u)}$ and 
$0>\iprod{(\upthing{H},h)}{(\upthing{u}^{\prime}\otimes 
\upthing{u}^{\prime},u)}$  hold. Note that in the $(d+1)\times(d+1)$ matrix 
$
(\upthing{u}^{\prime}\otimes\upthing{u}^{\prime}+\upthing{u}\otimes\upthing{u})
$, the first $d$ entries of the last row as well as of the last column are 0. 
Thus,
\[
0>\iprod{(\upthing{H},h)}{\big((\upthing{u}^{\prime}\otimes 
\upthing{u}^{\prime}+\upthing{u}\otimes\upthing{u})/2,u\big)}=
\]\[
\iprod{(\upthing{H}_0,h)}{\big((\upthing{u}^{\prime}\otimes 
\upthing{u}^{\prime}+\upthing{u}\otimes\upthing{u})/2,u\big)}=
\iprod{(\upthing{H}_0,h)}{(\upthing{u}\otimes\upthing{u},u)},
\]
completing the proof of \Href{Claim}{claim:HisinM}.
\end{proof}

In both parts of the proof of \Href{Theorem}{thm:johncond}, we will consider 
a path in $\ellips$ and  compute the derivative of the $s$-volume 
at the start of this path.

\begin{claim}\label{claim:svolumederivative}
Let $\varepsilon_0>0$ and let $\gamma:[0,\varepsilon_0]\to\Re$ be a continuous 
function whose right derivative at 0 
exists. Let $H\in\Re^{d\times d}$ be an arbitrary symmetric matrix and 
$h\in\Red$. Consider the path
\begin{equation}\label{eq:pathinE}
 \upthing{E}:[0,\varepsilon_0]\to\MM;\;\;t\mapsto
 \bigg(\upthing{I}+t\big(H\oplus \gamma(t)\big), th\bigg).
\end{equation}
For sufficiently small $t$, we have that $\upthing{E}(t)$ is in $\ellips$, and 
the right derivative of the $s$-volume is
\begin{equation}\label{eq:svolumederivative}
\left.\frac{\di}{\di t}\right|_{t=0^{+}}
\frac{\smeasure{\upthing{E}(t)}}{\volbs}= 
\iprod{(H\oplus\gamma(0),h)}{(\upthing{S},0)}.
\end{equation}
\end{claim}
\begin{proof}
We apply \eqref{eq:sfunc_of_ellipsoida}, 
\[
\left.\frac{\di}{\di t}\right|_{t=0^{+}}
 \frac{\smeasure{\upthing{E}(t)}}{\volbs}
 =
\left.\frac{\di}{\di t}\right|_{t=0^{+}}
\bigg[(1+t \gamma(t))^s\det(I+t H) \bigg]= 
\]\[
(1+0 \cdot \gamma(0))^s
\left.\frac{\di}{\di t}\right|_{t=0^{+}}
\bigg[\det(I+t H) \bigg]+ 
\det(I+0\cdot H)
\left.\frac{\di}{\di t}\right|_{t=0^{+}}
\bigg[(1+t \gamma(t))^s\bigg]=
\]\[
\tr{H}+s\gamma(0),
\]
which is equal to the right hand side of \eqref{eq:svolumederivative} 
completing the proof of \Href{Claim}{claim:svolumederivative}. 
\end{proof}

\begin{claim}\label{claim:HimpliesNotLocMax}
If there is $(H\oplus\gamma,h)\in\MM$ satisfying 
\eqref{eq:Hseparates}, then $\ball{d+1}$ is not a locally maximal $s$-volume 
\dsymm{}
ellipsoid contained in $\upthing{K}$.
\end{claim}
\begin{proof}
Let $\gamma(t)=\gamma$ be the constant function for $t\geq0$, and consider the 
path \eqref{eq:pathinE}. By \Href{Claim}{claim:svolumederivative} 
and \eqref{eq:Hseparates}, the $s$-volume has positive derivative at the start 
of this path. Clearly, $\smeasure{\upthing{E}(t)}$ is differentiable on some 
interval $[0,\varepsilon_0]$, and hence,
there is an $\epsilon_1>0$ such that for every 
$0< t <\varepsilon_1$, we have
\begin{equation}\label{eq:smeasurebig}
 \smeasure{\upthing{E}(t)} > \smeasure{\ball{d+1}}.
\end{equation}

Now, it suffices to establish that 
there is $\epsilon_2>0$ such that for all 
$0<t<\varepsilon_2$, we have
\begin{equation}\label{eq:variationsubset}
 \upthing{E}(t)\subseteq \upthing{K}.
\end{equation}

Set $\upthing{H}=H\oplus\gamma$.
First, we fix an arbitrary contact point $\upthing{u}\in C$. We 
claim that there is an $\varepsilon(\upthing{u})>0$ such that for every 
$0<t<\varepsilon(\upthing{u})$,
we have $(\upthing{I}+t \upthing{H})\upthing{u}+t h\in\inter{\ball{d+1}}$.
Indeed,
\[
 \iprod{(\upthing{I}+t 
\upthing{H})\upthing{u}+t h}{(\upthing{I}+t 
\upthing{H})\upthing{u}+t h}=
 1+2t 
\big(\iprod{\upthing{H}\upthing{u}}{\upthing{u}}+\iprod{h}{u}
\big)+o(t)=
\]\[
1+2t 
\iprod{(\upthing{H},h)}{(\upthing{u}\otimes\upthing{u},u)}+o(t).
\]
By \eqref{eq:Hseparates}, the latter is less than 1 for a sufficiently small 
positive $t$. 
Next, the compactness of $C$ yields that there is an $\varepsilon_3>0$ such 
that 
$(\upthing{I}+\varepsilon_3\upthing{H})C+\varepsilon_3 
h\subseteq\inter{\ball{d+1}}\subseteq\upthing{K}$.

By the continuity of the map $x\mapsto (\upthing{I}+\varepsilon_3 
\upthing{H})x+\varepsilon_3 h$, there is an open 
neighborhood $\mathcal{W}$ of $C$ in $\ball{d+1}$ such that 
$(\upthing{I}+\varepsilon_3\upthing{H})\mathcal{W}+\varepsilon_3 h
\subseteq\inter{\ball{d+1}}\subseteq\upthing{K}$.
The latter combined with $\mathcal{W}\subset\inter{\ball{d+1}}$ and with the 
convexity of $\ball{d+1}$ yield that for all $0<t<\varepsilon_3$, we 
have 
$(\upthing{I}+t\upthing{H})\mathcal{W}+t h
\subseteq\inter{\ball{d+1}}\subseteq\upthing{K}$.

On the other hand, the compact set
$\ball{d+1}\setminus\mathcal{W}$ is a subset of $\inter{\upthing{K}}$, and 
hence, there is an $\varepsilon_4>0$ such that for all
$0<t<\varepsilon_4$, we have 
$(\upthing{I}+t 
\upthing{H})(\ball{d+1}\setminus\mathcal{W})+t 
h\subseteq\inter{\upthing{K}}$. Thus, 
if $0< t < \min\{\varepsilon_3,\varepsilon_4\}$, then
$(\upthing{I}+t \upthing{H})(\mathcal{W})+th\subseteq\inter{\upthing{K}}$ and
$(\upthing{I}+t \upthing{H})(\ball{d+1}\setminus\mathcal{W})+th
\subseteq\inter{\upthing{K}}$. Thus, \eqref{eq:variationsubset} holds 
concluding the proof of \Href{Claim}{claim:HimpliesNotLocMax}.
\end{proof}

\subsection{Proof of part~(\ref{part:localToIsotropic}) of 
{Theorem}~\ref{thm:johncond}}
Assume that there are no contact points and positive weights satisfying 
\eqref{eq:johncond}.
By Claims~\ref{claim:Hseparates} and \ref{claim:HisinM}, there is 
$(H\oplus\gamma,h)\in\MM$ satisfying 
\eqref{eq:Hseparates}. \Href{Claim}{claim:HimpliesNotLocMax} yields that 
$\ball{d+1}$ is not a locally maximal $s$-volume ellipsoid contained in 
$\upthing{K}$.

The bound on $k$ follows from \Href{Claim}{claim:caratheodory}, completing the 
proof of
part~(\ref{part:localToIsotropic}) of \Href{Theorem}{thm:johncond}.

\subsection{Proof of part~(\ref{part:IsotropicToGlobal}) of 
{Theorem}~\ref{thm:johncond}}

Assume that there is an ellipsoid $\upthing{E}$, represented by 
$(A\oplus\alpha,a)$,
contained in 
$\inter{\slift{f}}$ 
with $\smeasure{\upthing{E}}>\smeasure{\ball{d+1}}$.


Set $G=A-I\in\Re^{d\times d}$, and define the function
$\gamma(t)=\frac{\alpha^{t}-1}{t}$ for $t\in(0,1]$, which, with 
$\gamma(0)=\ln\alpha$, is a continuous function on $[0,1]$ whose right 
derivative at 0 exists.
Consider the path
\begin{equation*}
 \upthing{E}:[0,1]\to\MM;\;\;t\mapsto
 \bigg(\upthing{I}+t\big(G\oplus \gamma(t)\big),ta\bigg).
\end{equation*}
Clearly, this path is in $\ellips$, it starts at $\upthing{E}(0)=\ball{d+1}$ and 
ends at 
$\upthing{E}(1)=\upthing{E}$. 

\begin{claim}\label{claim:HSbigeq}
\begin{equation}\label{eq:HSbigeq}
0\leq \iprod{(G\oplus\gamma(0),a)}{(\upthing{S},0)}.
\end{equation} 
\end{claim}
\begin{proof}
By \Href{Lemma}{lem:minkowski_for_s_ellipsoids}, for every $t\in[0,1]$, we have
\begin{equation*}
 \frac{\smeasure{\upthing{E}(t)}}{\volbs}\geq 1,
\end{equation*}
and hence, for the right derivative, we have
\begin{equation*}
\left.\frac{\di}{\di t}\right|_{t=0^{+}}
 \frac{\smeasure{\upthing{E}(t)}}{\volbs}\geq 0.
\end{equation*}
\Href{Claim}{claim:svolumederivative} now yields the assertion of 
\Href{Claim}{claim:HSbigeq}.
\end{proof}

We want to have strict inequality in \eqref{eq:HSbigeq}, thus we modify 
$G$ a bit. Let  
\[
H=G+\delta I, \text{ with a small }\delta>0.
\]
By Claim \ref{claim:HSbigeq}, we have
\begin{equation}\label{eq:HSbig}
0< \iprod{(H\oplus\gamma(0),a)}{(\upthing{S},0)}.
\end{equation}
Moreover, since $\upthing{A}\ball{d+1}+a\subset\inter{\slift{f}}$, we 
can fix $\delta>0$ sufficiently small such that we also have that
\begin{equation}\label{eq:Hisadmissible}
\left((I + H) \oplus (1 + \gamma(1)) \right) \ball{d+1} + a
 \subset  \inter{\slift{f}}.
\end{equation}

\begin{claim}\label{claim:supportplane}
Set $\upthing{H}_{0}=H\oplus\gamma(0)$. Then
\begin{equation}\label{eq:supportplane}
\iprod{(\upthing{H}_{0},a)}{(\upthing{u}\otimes\upthing{u},u)}\leq 0
\end{equation}
for every contact point $\upthing{u}\in C$.
\end{claim}
\begin{proof}
Fix an $\upthing{u}\in C$ and consider the curve $\xi: [0,1]\to\Redp; t\mapsto 
\upthing{u}+
t \big(H\oplus\gamma(t)\big)\upthing{u} +ta$ in $\Redp$.
By \Href{Lemma}{lem:inclusion_ellipsoids} and \eqref{eq:Hisadmissible}, the 
ellipsoid represented by $(\upthing{I},0)+t(H\oplus\gamma(t),a)$ is contained 
in 
$\slift{f}$ for every $t\in[0,1]$, and in particular, the curve $\xi$ is 
contained in $\slift{f}$.  By convexity and \eqref{eq:Hisadmissible}, we have 
that the projection of $\xi$ onto $\Red$ is a subset of the closure of the 
support of $f.$ Further, $\xi$ is a smooth curve and its tangent vector 
$\xi'(0)$ is given by
\[
\xi'(0) = \left.\frac{\di}{\di t}\right\vert_{t=0^{+}} 
 \left(\upthing{u}+t\big(H\oplus\gamma(t)\big)\upthing{u} +ta\right) = 
 \left.\frac{\di}{\di t}\right\vert_{t=0^{+}} 
 \left((tH\oplus (\alpha^t-1))\upthing{u} +ta\right) = (H\oplus \ln 
\alpha)\upthing{u} +a. 
\]

We consider two cases as to whether $\upthing{u} \in \Red$ or not.

First, if $\upthing{u} \in \Red,$ then $\upthing{u}$ belongs to the boundary of 
the support of $f.$  Since the support of a 
log-concave function is a convex set, we conclude that $\upthing{u}$  is the 
outer normal vector to the support of $f$ at $\upthing{u}.$ Thus, 
$\iprod{\xi'(0)}{\upthing{u}} \leq 0.$

Second, if $\upthing{u} \notin \Red,$ then  
\Href{Lemma}{lem:touching_cond_log_concave} 
implies that $\bd{\slift{f}}$ is a smooth  hypersurface in $\Redp$ 
at $\upthing{u}$, whose outer unit normal vector at $\upthing{u}$ is 
$\upthing{u}$ itself. Thus, the angle between  the tangent vector  vector 
$\xi'(0)$ of the curve $\xi$ 
and the outer normal vector of the hypersurface $\bd{\slift{f}}$ at 
$\upthing{u}$ is not acute. That is,
$\iprod{\xi'(0)}{\upthing{u}} \leq 0.$

Hence, in both cases, we have
\[
0 \geq 
\iprod{\xi'(0)}{\upthing{u}}
= 
\iprod{
 \left(H\oplus \ln \alpha)\upthing{u} +a\right)
 }{\upthing{u}},
\]
which is \eqref{eq:supportplane} completing the proof of 
\Href{Claim}{claim:supportplane}.
\end{proof}

In summary, \eqref{eq:HSbig} and \Href{Claim}{claim:supportplane} show that 
when $\left(\upthing{H}_0,a\right)$ is substituted in the place of 
$(\upthing{H},h)$, then \eqref{eq:Hseparatesweak} holds. Hence, by 
\Href{Claim}{claim:Hseparates}, the proof of 
part~\eqref{part:IsotropicToGlobal} 
of \Href{Theorem}{thm:johncond} is complete.

We rephrase \Href{Theorem}{thm:johncond} without any reference to lifting of a function to $\Redp$ as 
follows.

\begin{thm}\label{thm:johncondfunc}
Let $f$ be a proper  log-concave function on $\Red$, $s>0$.
 Assume  $\shf{\ball{d+1}}^s \leq f.$  Then the following are equivalent:
\begin{enumerate}
 \item
The function $\shf{\ball{d+1}}^s$ is the John $s$-function of $f.$
 \item
There are  points 
${u}_1,\ldots,{u}_k \in \ball{d}\subset \Red$ and 
positive weights $c_1,\ldots,c_k$ such that 
\begin{enumerate}
\item $f(u_i) = \shf{\ball{d+1}}^s(u_i)$ for all $i \in [k]$;
\item $\sum_{i=1}^k c_i {u}_i\otimes {u}_i ={I};$
\item $\sum_{i=1}^k c_i f^{1/s}({u}_i) \cdot f^{1/s}({u}_i) = s;$
\item $\sum_{i=1}^k c_i u_i =0,$
\end{enumerate}
where
${I}$ is the $d \times d$ identity matrix.
\end{enumerate}
\end{thm}

\section{Further inequalities and the limit as \tpdfs{} tends to 0}
\label{sec:furtherIneq}

\subsection{Comparison of the \tpdfs-volumes of John \tpdfs-ellipsoids for 
distinct values of \tpdfs}
\begin{lemma}\label{lem:comparison_sell_volume}
Let $f$ be a proper log-concave function on 
$\R^d$, and $0< s_1 < s_2$.
Then,
\[
\sqrt{
\left(\frac{s_2}{d + s_2}\right)^{s_2}
\left(\frac{d}{d+ s_2}\right)^{d}}
\cdot
\frac{\volbs[s_1]}{\volbs[s_2]}
  \leq
   \frac{\smeasure[s_1]{\sellbody[s_1]{f}}}{\smeasure[s_2]{\sellbody[s_2]{f}}}
   \leq
   \frac{\volbs[s_1]}{\volbs[s_2]}.
\] 

\end{lemma}
\begin{proof}
We start with the second inequality.
We may assume that $\sellbody[s_1]{f} = \ball{d+1}$, and hence,
its height function is 
$\shf{\sellbody[s_1]{f}} (x)= \sqrt{1 - |x|^2}$ 
for $x \in \ball{d}.$ 
Since $s_1< s_2$ and $\shf{\sellbody[s_1]{f}} (x) \le 1,$ we have
\[
\left(\shf{\sellbody[s_1]{f}} (x) \right)^{s_2} \le
\left( \shf{\sellbody[s_1]{f}} (x) \right)^{s_1}
 \le f(x) \quad \text{ for all } x \in \ball{d}.
\]
That is, by \eqref{eq:height_relation}, $ \ball{d+1} \subset 
\slift[s_2]{f}$, which yields
$\smeasure[s_2]{\ball{d+1}}\leq 
\smeasure[s_2]{\sellbody[s_2]{f}}$. Hence,
\[
   \frac{\smeasure[s_1]{\sellbody[s_1]{f}}}{\smeasure[s_2]{\sellbody[s_2]{f}}} 
\leq
\frac{\smeasure[s_1]{\ball{d+1}}} 
{\smeasure[s_2]{\ball{d+1}}} 
{=}
\frac{\volbs[s_1]}{\volbs[s_2]}.
\]
{Next, we prove the first} inequality of the assertion of the lemma.
Now, we assume that $\sellbody[s_2]{f} = \ball{d+1}.$
Therefore, for any ${\rho} \in (0,1)$, we have that $\slift[s_2]{f}$ 
contains the cylinder $\rho \ball{d} \times [-\sqrt{1-\rho^2},\sqrt{1-\rho^2}].$ 
Hence, 
$\slift[s_1] {f}$ contains the ellipsoid $\upthing{E},$ represented by  
$\left(\rho I \oplus \left(\sqrt{1-\rho^2}\right)^{s_2 /s_1}, 0 \right)$, whose 
$s_1$-volume by \eqref{eq:sfunc_of_ellipsoida} is $\volbs[s_1] \cdot \rho^d  
\cdot \left(1-\rho^2\right)^{s_2 / 2}.$ Choosing $\rho = 
\sqrt{\frac{d}{d+s_2}},$  we obtain
\[
\sqrt{\left(\frac{s_2}{d + s_2}\right)^{s_2}\left(\frac{d}{d+ s_2}\right)^{d}} 
\cdot 
\frac{\volbs[s_1]}{\volbs[s_2]} 
{=} 
\frac{\smeasure[s_1]{\upthing{E}}}{\smeasure[s_2]{\ball{d+1}}} \leq
   \frac{\smeasure[s_1]{\sellbody[s_1]{f}}}{\smeasure[s_2]{\sellbody[s_2]{f}}}.
\]
\end{proof}

\subsection{Stability of the John \tpdfs-ellipsoid}

\begin{lemma}
\label{lem:stability_volume_s-ellipsoid}
Fix a dimension $d$ and a positive constant $C > 0.$ 
Then there exist  constants $\epsilon_C > 0$ and $k_C > 0$ with the following 
property.  Let $s \in (0, \infty),$
$\epsilon \in [0, \epsilon_C]$ and $f$ be a proper log-concave function on 
$\Red$, whose \jsellipsoid{} $\sellbody{f}$ is represented by    
$(A_1 \oplus \alpha_1, a_1)$,  and let $\upthing{E}_2$ denote another 
ellipsoid, represented by $(A_2 \oplus \alpha_2, a_2)$, with $\upthing{E}_2 
\subset 
\slift{f}$. Assume that
\begin{equation}
\label{eq:lem_stability_C_bound}
 \smeasure{\sellbody{f}} \ge C - \epsilon 
 \quad \text{and} \quad
    \smeasure{\sellbody{f}}  \ge \smeasure{\upthing{E}_2} \ge  
\smeasure{\sellbody{f}} - \epsilon.
\end{equation}
Then
\begin{equation}
\label{eq:lem_stability_theta_bound}
\norm{\frac{A_1}{\norm{A_1}} - \frac{A_2}{\norm{A_2}}} + \frac{|\alpha_1^s - 
\alpha_2^s|}{\norm{f}} + \frac{|a_1 - a_2|}{\norm{A_1} \cdot \norm{f}} \le k_C 
\sqrt{\epsilon}.
\end{equation}
\end{lemma}

In this subsection, we prove \Href{Lemma}{lem:stability_volume_s-ellipsoid}.

Let 
$\upthing{E}$ denote the ellipsoid  represented by
\[
\left(\frac{A_1+A_2}{2} \oplus \sqrt{\alpha_1 \alpha_2}, \frac{a_1+a_2}{2} 
\right).
\]
By Lemma \ref{lem:inclusion_ellipsoids}, $\upthing{E} \subset \slift{f},$
and, therefore, $\smeasure{\sellbody{f}} \ge \smeasure{\upthing{E}}.$

\begin{claim}\label{claim:stability_norm}
There are 
constants $\epsilon_0 > 0$ and $k_0 > 0$ such that 
if the ellipsoids $\sellbody{f}$ and $\upthing{E_2}$ satisfy 
\eqref{eq:lem_stability_C_bound} for  $\epsilon \in [0,\epsilon_0],$ then 
\begin{equation}\label{eq:stability_operator_multi}
(1 - k_0 \sqrt{\epsilon}) A_1 \prec A_2  \prec (1 + k_0 \sqrt{\epsilon}) A_1,
\end{equation}
and
\[\norm{\frac{A_1}{\norm{A_1}} - \frac{A_2}{\norm{A_2}}} \leq k_0 
\sqrt{\epsilon}
\quad \text{and} \quad
1 - k_0 \sqrt{\epsilon} \leq \frac{\det A_1}{\det A_2} \leq 
1 + k_0 \sqrt{\epsilon}.
\]
\end{claim}
\begin{proof}
By \eqref{eq:sfunc_of_ellipsoida}, we have
\[
\frac{\smeasure{\upthing{E}}}{\sqrt{\smeasure{\sellbody{f}}\smeasure{\upthing{E
}_2}}} = \frac{1}{2^d}\frac{\det (A_1 + A_2)}{\sqrt{\det A_1 \det A_2}}.
\]
Since $\smeasure{\sellbody{f}} \ge \smeasure{\upthing{E}}$  and by 
\eqref{eq:lem_stability_C_bound}, 
there exist 
$\epsilon_1 > 0$ and $k_1 > 0$ such that the left-hand side in the equation 
above is at most 
$1 + k_1 \cdot  \epsilon$ for all $\epsilon \in [0, \epsilon_1].$
 Therefore, we have that
 \begin{equation}
 \label{eq:det_stupid_bound}
 1 + k_1 \cdot  \epsilon \ge \frac{1}{2^d}\frac{\det (A_1 + A_2)}{\sqrt{\det 
A_1 \det A_2}}.
 \end{equation}
Let $R$ be the square root of $A_1$, and $U$ be the orthogonal transformation 
that diagonalizes $R^{-1}A_2R^{-1},$ that is, the matrix $D = U R^{-1}A_2R^{-1} 
U^T$ is diagonal. Let $D = \operatorname{diag}({\beta_1, \dots, 
\beta_d}).$ Then for $S = 
UR^{-1},$ we have 
$S A_1 S^T = I, S A_2 S^T = D.$ By the multiplicativity of the 
determinant,
inequality \eqref{eq:det_stupid_bound} is equivalent to
\[
1 + k_1 \cdot  \epsilon \ge \prod\limits_{1}^{d}\frac{1 + \beta_i}{2 
\sqrt{\beta_i}}.
\]
Since $1 + \beta \ge 2 \sqrt{\beta}$ for any $\beta>0$, 
this implies that 
\[
1 + k_1 \cdot  \epsilon \ge \frac{1 + \beta_i}{2 \sqrt{\beta_i}}
\]
for every $i \in [d]$.
If we consider the above formula as a quadratic inequality in the variable 
$\sqrt{\beta_i}$, then, by the quadratic formula, we obtain that there exist 
positive constants
$k_2$ and $\epsilon_2$ such that the inequality  
\begin{equation}\label{eq:betabound}
1 - k_2 \sqrt{\epsilon}  \leq \beta_i \leq 1 +k_2 \sqrt{\epsilon} 
\end{equation}
holds for all $\epsilon \in [0, \epsilon_2].$

{Clearly, $\frac{\det A_1}{\det A_2}=1/\prod_{i=1}^d \beta_i$ and hence,}
the estimate on $\frac{\det A_1}{\det A_2}$ follows from \eqref{eq:betabound}.

On the other hand, \eqref{eq:betabound} yields also that 
\[(1 - k_2 \sqrt{\epsilon}) I \prec {S}A_2 {S^T} \prec (1 + k_2 
\sqrt{\epsilon}) I.
\]
Thus, \eqref{eq:stability_operator_multi} follows.
Hence, there exist  positive constants $k_3$ and $\epsilon_3$ such that the 
inequality
\[
\left|\frac{\norm{A_2}}{\norm{A_1}} - 1 \right| \leq 
k_3 \sqrt{\epsilon}
\] 
holds for all $\epsilon \in [0, \epsilon_3].$
This and \eqref{eq:stability_operator_multi} yield that
\[
\norm{\frac{A_1}{\norm{A_1}} - \frac{A_2}{\norm{A_2}}} \leq 
\norm{\frac{A_1}{\norm{A_1}} - \frac{A_2}{\norm{A_1}}} + 
\norm{\frac{A_2}{\norm{A_1}} - \frac{A_2}{\norm{A_2}}} \leq (k_2 + k_3) 
\sqrt{\epsilon}
\]
for all $\epsilon \in [0, \min\{\epsilon_2, \epsilon_3\}].$
This completes the proof of Claim~\ref{claim:stability_norm}.
\end{proof}

\begin{claim}\label{claim:stability_height}
There are constants $\epsilon_0 > 0$ and $k_0 > 0$ such that 
if the ellipsoids $\sellbody{f}$ and $\upthing{E_2}$ satisfy 
\eqref{eq:lem_stability_C_bound} for  $\epsilon \in [0,\epsilon_0],$ then
\[|\alpha_1^s - \alpha_2^s| \leq \norm{f} k_0 \sqrt{\epsilon}.
\]
\end{claim}
\begin{proof}
By identity \eqref{eq:sfunc_of_ellipsoida} and the inequalities 
\eqref{eq:lem_stability_C_bound},
we have that
\[
1 \ge 
\frac{\smeasure{\upthing{E}_2}}{\smeasure{\sellbody{f}}} = 
\frac{\det A_2 \cdot \alpha_2^s}{ \det A_1 \cdot \alpha_1^s} \ge 1 - 
\frac{\epsilon}{\smeasure{\sellbody{f}}}.
\]
By this and by \Href{Claim}{claim:stability_norm}, 
we get the following inequality
\[
\left(1 + k_1 \sqrt{\epsilon} \right) \alpha_1^s \ge \alpha_2^s \ge 
\left(1 - k_1 \sqrt{\epsilon}\right) \alpha_1^s
\]
for all $\epsilon \in [0, \epsilon_1],$ where 
$k_1$ and $\epsilon_1$ are some positive constants.
Equivalently, we have that

\[
k_1 \sqrt{\epsilon} \cdot \frac{ \alpha_1^s}{\norm{f}} \ge \frac{\alpha_2^s - 
\alpha_1^s}{\norm{f}} \ge 
 -k_1 \sqrt{\epsilon} \cdot \frac{\alpha_1^s}{\norm{f}}.
\]
The claim follows since $ \frac{\alpha_1^s}{\norm{f}}  \leq 1.$
\end{proof}
To complete the proof of \Href{Lemma}{lem:stability_volume_s-ellipsoid}, we 
need 
to show that $a_1$ and $a_2$ are close.
By translating the origin and rotating the space $\Red$, we may assume that 
$a_1  = - 
a_2 \neq 0$ and that $A_1^{-1} a_1 =\delta e_1$ for some $\delta>0$.
Consider the origin centered \dsymm{} ellipsoid 
\[
\upthing{E}_0 =  (A_1 \oplus {\alpha_1})  \upthing{M} \ball{d+1}, \quad 
\text{where} \;\; 
\upthing{M} = \diag(1+\delta, 1, \dots, 1).
\]
Clearly, 
$\smeasure{\upthing{E}_0}= \smeasure{\sellbody{f}}
\left(1 + \frac{|A_1^{-1}(a_1 - a_2)|}{2}\right) \geq 
\smeasure{\sellbody{f}}
\left(1 + \frac{|(a_1 - a_2)|}{2 \norm{A_1}}\right)$.

{
By \eqref{eq:stability_operator_multi} and 
\Href{Claim}{claim:stability_height}, we have 
\[
\left((1 - k_0 \sqrt{\epsilon}) A_1 
\oplus 
(1 - k_0 \norm{f}\sqrt{\epsilon})^{1/s} 
\alpha_1\right)\ball{d+1} + a_2\subseteq \slift{f}.
\]
On the other hand, clearly,
\[
\left((1 - k_0 \sqrt{\epsilon}) A_1 
\oplus 
(1 - k_0 \norm{f}\sqrt{\epsilon})^{1/s} 
\alpha_1\right)\ball{d+1} + a_1\subseteq
\left(A_1 
\oplus \alpha_1\right) \ball{d+1} + a_1\subseteq \slift{f}.
\]

Thus, by \Href{Lemma}{lem:interpolation_translates},
\[
\left((1 - k_0 \sqrt{\epsilon}) A_1 
\oplus 
(1 - k_0 \norm{f}\sqrt{\epsilon})^{1/s} 
{\alpha_1}\right)\upthing{M} \ball{d+1}  
\]
is contained in $\slift{f}$.

}

By \eqref{eq:sfunc_of_ellipsoida}, the $s$-volume of this ellipsoid is 
 \[ \smeasure{\sellbody{f}} \left( 1 - k_0 \norm{f}\sqrt{\epsilon} \right) 
\left( 1 - k_0 \sqrt{\epsilon} \right)^{d} \left(1 + \frac{|(a_1 - a_2)|}{2 
\norm{A_1}}\right) \leq \smeasure{\sellbody{f}}.
 \]
 Thus, there exist constants $\epsilon_1, k_1 > 0$ such that 
  $\frac{|a_1 - a_2|}{\norm{A_1}} \le k_1 \norm{f} \sqrt{\epsilon}$ for any 
$\epsilon \in [0, \epsilon_1].$  From this and Claims 
\ref{claim:stability_norm} 
and 
\ref{claim:stability_height}, \Href{Lemma}{lem:stability_volume_s-ellipsoid} 
follows.

\subsection{The limit as \texorpdfstring{$s \to 0$ }{s --> 0}}
We recall from \Href{Section}{sec:intro} the approach of 
Alonso-Guti{\'e}rrez, Gonzales Merino, Jim{\'e}nez and Villa 
\cite{alonso2018john}.

Let $f$ be a proper log-concave function on $\Red$. For every 
$\beta \in (0, \norm{f})$, 
consider the superlevel set $[f\geq\beta]$ of $f$. This is a bounded convex set 
with non-empty interior in $\Red$, we take its largest volume ellipsoid, and 
multiply the volume of this ellipsoid by $\beta$. As shown in  
\cite{alonso2018john}, 
there is a unique $\beta_0 \in \left[0, \norm{f}\right]$ 
such that this product is maximal. Furthermore, $\beta_0 \geq e^{-d}\norm{f}$.
We call the ellipsoid $E$ in $\Red$ obtained for this $\beta_0$ 
\emph{\bernardoell{}}. 

We  refer to a function of the form $\beta \chi_{E},$ where $E \subset \R^d$ is 
an ellipsoid in $\Red$ and  $\beta > 0,$
as a \emph{flat ellipsoid function}. 
We will say that $(A \oplus \alpha, a) \in \ellips$  represents the 
flat ellipsoid function $\alpha \chi_{A \ball{d} +a}.$
Clearly, any flat ellipsoid function is represented by a unique element of 
$\ellips$ and \bernardoell{} is the maximal integral flat ellipsoid 
function 
among all flat ellipsoid functions that are below $f.$


\begin{thm}
[\Bernardoell{} is the John $0$-ellipsoid]
\label{thm:getbackbernardo}
Let $f$ be a proper log-concave function. Then
there exists  
$(A\oplus \alpha, a) \in \ellips$ such that
\begin{enumerate}
\item\label{part:1_limit_s_0}
The function $\alpha \chi_{A \ball{d} +a}$ is below $f.$
\item\label{part:2_limit_s_0}  The functions 
$
\selldense[s]{f} 
$
converge uniformly to $\alpha \chi_{A \ball{d} 
+a}$ on the complement of any open neighborhood of the boundary 
in $\R^d$ of $A \ball{d} + a$ as $s$ tends to $0$.
\item \label{part:3_limit_s_0}
The function $\alpha \chi_{A \ball{d} +a}$ is the unique flat ellipsoid 
function 
of maximal integral among all flat ellipsoid functions that are below $f.$
\end{enumerate}
\end{thm}

In this subsection, we prove \Href{Theorem}{thm:getbackbernardo}.

We start with the existence of the limit flat ellipsoid function in 
\eqref{part:2_limit_s_0}.
Let   $\sellbody{f}$ be represented
by $(A_s \oplus \alpha_s, a_s)$ for every $s \in (0,1].$ 
\begin{claim}\label{claim:limit_zero_existence}
The following limits exist
\begin{equation}
\label{eq:limits_s_zero}
 \lim\limits_{s \to 0^{+}} \smeasure{\sellbody{f}}= \mu > 0, \,
 \lim\limits_{s \to 0^{+}}{A_{s}} =  A,  \;
 \lim\limits_{s \to 0^{+}} \alpha_s^s = \alpha > 0
\quad \text{and} \quad \lim\limits_{s \to 0^{+}} a_{s} = a,
\end{equation}
where $A$ is positive definite.
\end{claim}
\begin{proof}
Since the John $1$-ellipsoid exists, by \eqref{eq:kappaszero} and  
\Href{Lemma}{lem:comparison_sell_volume}, 
we have 
\[\inf\limits_{s \in (0,1]} \smeasure{\sellbody{f}} > 0.\]
 Recall from \Href{Section}{sec:basicboundsellipsoid} that $\volbs$ as a 
function of $s$ with $\volbs[0] = \vol{d} \ball{d}$  is a strictly decreasing 
continuous function  on  $[0, \infty)$. 
Applying \Href{Lemma}{lem:boundedness} 
with $g=\shf{\ball{2}}^s$ for each $s \in (0,1]$, where $\shf{\ball{2}}$ is the 
height function of $\ball{2}$ (see \eqref{eq:heightball}), 
we conclude that for some positive constants $\vartheta, \rho, \rho_1, \rho_2$  
and all $s \in (0,1]$, the inequalities 
$ \vartheta \leq \alpha_s \leq \norm{f}$ and  
$|a_s| \leq  \rho,$ 
and  the relation $\rho_1 I \prec A_s \prec \rho_2 I$ hold. 
Thus, there exists a sequence of positive reals $\{s_i\}_1^{\infty}$ with 
$\lim\limits_{i \to \infty} s_i = 0$ such that 
\[
 \smeasure[s_i]{\sellbody[s_i]{f}} \to  \limsup\limits_{s \to 0^{+}} 
\smeasure{\sellbody{f}}, \,
 {A_{s_i}} \to  A  , \;
\alpha_{s_i}^{s_i} \to \alpha
\quad \text{and} \quad a_{s_i}  \to a
\]
{for a positive definite matrix $A\in\Re^{d\times d}$,  
$\alpha>0$ and $a\in\Red$,} as $i$ tends to $\infty.$
definite.

We use $J_f$ to denote the  flat ellipsoid function represented by $(A \oplus 
\alpha, a).$  
Clearly, $J_f$ is below $f.$
Consider the ellipsoids $\upthing{E}_s$ represented by 
$(A \oplus 
\alpha^{1/s}, a)$ for all $s \in (0,1].$ 
Then, $\upthing{E}_s \subset \slift{J_f} \subset 
\slift{f}$ for every $s \in (0,1].$ 
By 
\eqref{eq:sfunc_of_ellipsoida} and \eqref{eq:kappaszero}, we have 
\[
\smeasure{\upthing{E}_s}=
\frac{\volbs}{\vol{d} \ball{d}} \int_{\Red} J_f \to
\int_{\Red} J_f \text{ as } s \to 0^{+}.
\]
That is,  $\lim\limits_{s \to 0^{+}} \smeasure{\sellbody{f}} = \int_{\Red} 
J_{f} 
\di 
x.$ 
As an immediate consequence,  \Href{Lemma}{lem:stability_volume_s-ellipsoid} 
implies   \eqref{eq:limits_s_zero}.
\end{proof}

\begin{claim}\label{claim:zero_is_bernardo}
$J_f$, {as defined in the proof of 
\Href{Claim}{claim:limit_zero_existence}}, is the unique flat ellipsoid 
function that is of maximal integral among those that are below $f$.
\end{claim}
\begin{proof}
Assume that there is a flat ellipsoid function $J_E,$ represented by $(A_0 
\oplus 
\alpha_0, a_0),$ such that
$\int_{\Red} J_E \geq \int_{\Red} J_{f}.$
Consider the ellipsoids $\upthing{E}'_s$ represented by $(A_0 \oplus 
\alpha_0^{1/s}, a_0)$ for all $s \in (0,1].$ Clearly, $\upthing{E}'_s \subset 
\slift{f}$ for every $s \in (0,1].$ 
By 
\eqref{eq:sfunc_of_ellipsoida}, 
$\smeasure{\upthing{E}'_s}=
\frac{\volbs}{\vol{d} \ball{d}} \int_{\Red} J_E$.
By \eqref{eq:kappaszero} and by the definition of the \jsellipsoid, we 
have that
\[
\int_{\Red} J_E= 
\lim\limits_{s \to 0^{+}} \smeasure{\upthing{E}'_s}
\leq 
\lim\limits_{s \to 0^{+}} \smeasure{\sellbody{f}}
=
\int_{\Red} J_f.
\]
{Thus, for every positive integer $i,$ there is $s_i > 0$ such that}  
\[
\smeasure{\upthing{E}'_s}  \geq \smeasure{\sellbody{f}} - \frac{1}{i}  \geq 
{\int_{\Red}} J_f - \frac{2}{i}\] 
for all $s \in (0, s_i].$
Finally, by \Href{Lemma}{lem:stability_volume_s-ellipsoid},
we have that $\lim\limits_{s_i \to 0^{+}} A_0 = A, 
\lim\limits_{s_i \to 0^{+}} \alpha_0 = \alpha$ and $\lim\limits_{s_i \to 0^{+}} 
a_0 = 
a.$ That is, $J_f$ and $J_E$ coincide.
\end{proof}

\Href{Theorem}{thm:getbackbernardo} is an immediate consequence of Claims 
\ref{claim:limit_zero_existence} and \ref{claim:zero_is_bernardo}.

\subsection{Integral ratio}
For any $s \in [0, \infty)$ and positive integer $d,$ it is reasonable  to 
define the \emph{$s$-integral 
ratio} of $f : \Red \to [0,\infty)$ by
\[
\sintrat(f) = 
\left( 
\frac{\int_{\Red} f}{\int_{\Red} \selldense{f}}
\right)^{1/d}.
\]

Corollary~1.3 of \cite{alonso2018john} states that 
there exists  $\Theta > 0$ such that 
 \[
\sintrat[0](f) \leq \Theta \sqrt{d},
\]
for any proper log-concave function $f: \R^d \to [0,\infty)$ and any positive 
integer $d.$

{Using  \Href{Lemma}{lem:comparison_sell_volume} and}  
\Href{Theorem}{thm:getbackbernardo}, we obtain the following.

\begin{cor}\label{cor:volume_ratio} 
Fix $s \in [0, \infty).$ Then there exists $\Theta_s$ such that for any positive 
integer
 $d$  and any proper log-concave function $f: \Red \to [0,\infty)$, the 
following inequality holds.
 \[
\sintrat[s](f) \leq \betaf{s/2 +1}{d/2}^{-\frac{1}{d}} \cdot \sintrat[0](f) 
 \leq \Theta_s 
\sqrt{d},
\]
where $\betaf{\cdot}{\cdot}$ denotes Euler's Beta function.
\end{cor}
\section{Large \tpdfs{} behavior}
\label{sec:sinfty}
We will say that a Gaussian density on $\Red$ defined by $x\mapsto \alpha  e^{- 
\iprod{A^{-1}(x-a)}{A^{-1}(x-a)}}$ is \emph{represented} by  
$(A \oplus \alpha, a) \in \ellips.$
  Clearly, any Gaussian density is represented by a unique element of 
$\ellips.$ We will denote 
the Gaussian density represented by $(A \oplus \alpha, a)$
as $G[(A \oplus \alpha, a)].$ 
If a Gaussian density is represented by  $(A \oplus \alpha, a) \in \ellips,$ we 
will call 
$\alpha$ its \emph{height}. We have that
\begin{equation}
\label{eq:gaussian_volume}
\int_{\Red} G[(A \oplus \alpha, a)] = \alpha \pi^{d/2} \det A.
\end{equation}
We will need the following property of Euler's Gamma function
(see \cite[6.1.46]{abramowitz1948handbook})
\begin{equation*}
\lim\limits_{s \to \infty} \frac{\Gamma(s +t_1)}{\Gamma(s +t_2)} s^{t_2-t_1} = 
1.
\end{equation*}
Using this in \eqref{eq:kappasfirst}, we obtain
\begin{equation}
\label{eq:volbs_limit_at_infinity}
\lim\limits_{s \to \infty}  \volbs \cdot \left(\frac{s}{2}\right)^{d/2}= 
  \pi^{d/2}.
\end{equation}
\subsection{Existence of a maximal Gaussian}

\begin{thm}\label{thm:infinity_ellipsoid}
Let $f:\Red\to[0,\infty)$ be a proper log-concave function. 
If there is a Gaussian density below $f,$ then there exists a  
Gaussian density below $f$ of 
maximal integral. All Gaussian densities of  maximal integral below  $f$ are 
translates of each 
other.
\end{thm}
\begin{proof}[Proof of Theorem \ref{thm:infinity_ellipsoid}]
The proof mostly repeats the argument in \Href{Section}{sec:interpolation}.

\Href{Lemma}{lem:boundedness} with $g = e^{-t^2},$ $t \in \R,$ implies that if 
there exists a Gaussian density below $f$, then
there is a Gaussian density of maximal integral among those that are below $f$. 
Next, we show that this Gaussian density of maximal integral is unique up to 
translation.

First, we need the following extension of Lemmas 
\ref{lem:inclusion_ellipsoids} and \ref{lem:minkowski_for_s_ellipsoids}.

\begin{lemma}[Interpolation between Gaussians]\label{lem:inclusion_gaussians}
Fix $
\beta_1,\beta_2 > 0$ with
$\beta_1  + \beta_2 = 1$.   Let $f_1$ and $f_2$ be two proper log-concave 
functions on $\Red$, and 
$G_1, G_2$ be two Gaussian densities represented 
by $(A_1 \oplus \alpha_1, a_1) \in \ellips$ and $(A_2 \oplus \alpha_2, a_2) \in 
\ellips,$ respectively, such that 
$G_1 \leq f_1$  and $G_2 \leq f_2.$ 
With the operation introduced in \Href{Section}{sec:asplund}, define
\[
f = \loginfconv{(\beta_1 \ast f_1)}{(\beta_2 \ast f_2)},
\]
and set
\[
(A \oplus \alpha, a) = \big((\beta_1 A_1 + \beta_2 A_2) \oplus 
\alpha_1^{\beta_1}\alpha_2^{\beta_2}, \beta_1 a_1 + \beta_2 a_2\big) .
\]
Then, $G\left[(A \oplus \alpha, a)\right] \leq f$  and the 
following inequality holds 
\begin{equation}\label{eq:minkowski_for_gaussians}
\int_{\Red} G[(A \oplus \alpha, a)] \ge \left(\int_{\Red} G_1 \right)^{\beta_1}
\left(\int_{\Red} G_2\right)^{\beta_2},
\end{equation}
with equality if and only if $A_1 = A_2.$
\end{lemma}
\begin{proof}
Fix $x \in \Red$ and define $x_1, x_2$ by 
\begin{equation}
\label{eq:check_contaiment_relation}
x_1 - a_1 =  A_1 A^{-1} (x - a), \quad x_2 -a_2 =  A_2 A^{-1} 
(x - a). 
\end{equation}

Since $G_1\leq f_1, G_2\leq f_2$,  we have  
\begin{equation}\label{eq:mink_gaussian_heights}
f_1 (x_1) \ge \alpha_1 e^{- \iprod{A_1^{-1}(x_1 - a_1)}{A_1^{-1}(x_1 - a_1)}}
\quad \text{and} \quad 
f_2 (x_2) \ge \alpha_2 e^{- \iprod{A_2^{-1}(x_2 - a_2)}{A_2^{-1}(x_2 - a_2)}}.
\end{equation}
Since  
$ \beta_1 x_1  + \beta_2 x_2   = x$ and  by  the definition of the Asplund sum, 
  we  have that
\[
f(x)    \ge f_1^{\beta_1}( x_1) f_2^{\beta_2}( x_2). 
 \]
Combining this with inequalities \eqref{eq:mink_gaussian_heights} and 
\eqref{eq:check_contaiment_relation}, we obtain
\begin{gather*}
{f(x)} \geq 
 \alpha_1^{\beta_1} \alpha_2^{\beta_2} 
 e^{- \beta_1 \iprod{A_1^{-1}(x_1 - a_1)}{A_1^{-1}(x_1 - a_1)}}   e^{- \beta_2 
\iprod{A_2^{-1}(x_2 - a_2)}{A_2^{-1}(x_2 - a_2)}} =\\
   \alpha_1^{\beta_1} \alpha_2^{\beta_2} 
   e^{-(\beta_1 + \beta_2) \iprod{A^{-1}(x - a)}{A^{-1}(x - a)}}  = 
\alpha_1^{\beta_1} \alpha_2^{\beta_2} 
   e^{- \iprod{A^{-1}(x - a)}{A^{-1}(x - a)}}.
 \end{gather*}
Thus, $G$ is below $f.$

We proceed with showing \eqref{eq:minkowski_for_gaussians}.
Substituting \eqref{eq:gaussian_volume}, inequality
\eqref{eq:minkowski_for_gaussians} takes the form
\[
 \pi^{d/2} \alpha_1^{\beta_1} \alpha_2^{\beta_2} \cdot \det{\left(\beta_1 A_1 + 
\beta_2 A_2 \right)} \ge
\pi^{d/2} \alpha_1^{\beta_1} \alpha_2^{\beta_2} \cdot \left(\det A_1 
\right)^{\beta_1} \left(\det A_2\right)^{\beta_2},
\]
or, equivalently, 
\[
 \det{\left(\beta_1 A_1 + \beta_2 A_2 \right)} \ge
\left(\det A_1 \right)^{\beta_1} \left(\det A_2\right)^{\beta_2}.
\]
Thus, inequality \eqref{eq:minkowski_for_gaussians} and its equality condition
follow from  Minkowski's determinant inequality 
\eqref{eq:minkowski_det_multipl_ineq} and the equality condition 
therein, completing the proof of \Href{Lemma}{lem:inclusion_gaussians}.
\end{proof}
Let $G_1$, represented by $(A_1  \oplus \alpha_1, a_1)$,     be a maximal 
integral Gaussian density that is below $f.$ 
 Assume that there is another Gaussian density $G_2$, represented by 
 $(A_{2} \oplus \alpha_2, a_2)$,  below $f$ with the same integral as
 $G_1.$
 Consider the Gaussian density $G$ represented by
 \[
\left(\frac{A_1+A_2}{2} \oplus \sqrt{\alpha_1 \alpha_2}, \frac{a_1+a_2}{2} 
\right) \in \ellips.
\]
By \eqref{eq:loginvconv_same_f} and \Href{Lemma}{lem:inclusion_gaussians},
we have that $G$ is below $f.$
Next, by the choice of the Gaussian densities, we also have
\[
\int_{\Red}{G} \leq 
\int_{\Red} {G_1} =
\sqrt{\int_{\Red} {G_1} \int_{\Red}{G_2}} =
\int_{\Red} {G_2},
\]
which, combined with 
 \Href{Lemma}{lem:inclusion_gaussians}, yields
\[
\int_{\Red}{G} = 
\int_{\Red} {G_1} =
\int_{\Red} {G_2}, 
\text{ and } A_1 = A_2.
\]
Combined with \eqref{eq:gaussian_volume}, it implies $\alpha_1 = \alpha_2$.
This completes the proof of \Href{Theorem}{thm:infinity_ellipsoid}.
\end{proof}

\subsection{Uniqueness does not hold for 
\texorpdfstring{$s=\infty$}{s=infinity}}\label{sec:non_unique_gaussian}

In this subsection, first, we show that it is possible that two Gaussian 
densities $G[(A 
\oplus \alpha, a_1)]$ and $G[(A \oplus \alpha, a_2)]$ with $a_1 \neq a_2$ below 
a proper log-concave function $f$ are  
of maximal integral. Next in \Href{Proposition}{prop:gaussian_uniqeness}, we 
show that uniqueness holds for a certain important 
class of log-concave functions.

Consider the Asplund sum 
\[
f = \loginfconv{G[(A \oplus \alpha, a)]}{\chi_{K}},
\]
where $(A \oplus \alpha, a) \in \ellips$ and $K$ is a compact convex set in 
$\Red.$ We claim that the set of the maximal
integral Gaussian densities that are below $f$ is
\[
\left\{ 
G[(A \oplus \alpha, a_m)] \st a_m \in a + K
\right\}.
\] 
To see this, one observes that if $G[(A^\prime \oplus \alpha^\prime , 
a^\prime)]\leq f$, then $A^\prime\preceq A$. The claim now follows from 
\eqref{eq:gaussian_volume}.

Uniqueness of the maximal Gaussian density below $f$ holds for an 
important class of log-concave functions.

\begin{prp}\label{prop:gaussian_uniqeness}
Let $K \subset \Red$ be a compact convex set containing the origin in the 
interior, and let 
$\norm{\cdot}_{K}$ denote the gauge function of $K$, that is, 
$\norm{x}_{K}=\inf\{\lambda>0\st x\in\lambda K\}$. Let $A\! \left(\ball{d} 
\right)$ be the largest volume origin centered ellipsoid contained in $K$, 
where 
$A$ is a positive definite matrix. Then the Gaussian density represented by 
$\left(A \oplus 1, 0 \right)$ is the unique maximal integral Gaussian density 
below the log-concave function $e^{-\norm{x}_K^2}.$
\end{prp}
\begin{proof}
Let $(A^\prime \oplus \alpha^\prime, a^\prime)\in\ellips$ be such that 
$G[(A^\prime \oplus \alpha^\prime, a^\prime)]\leq f$. 
First, we show that $A^\prime\ball{d}\subseteq K$. Indeed, we have
\[
\iprod{(A^\prime)^{-1}(x-a^\prime)}{(A^\prime)^{-1}(x-a^\prime)}
-\ln(\alpha^\prime)
\geq \norm{x}_K^2 
\]
for every $x\in\Red$. Suppose for a contradiction that there is a $y\in 
A^\prime\left(\inter{\ball{d}}\right)\setminus K$. Consider $x=\vartheta y$, 
and 
substitute into the previous inequality. We obtain 
\[
\vartheta^2|(A^\prime)^{-1}y|^2-2\vartheta
\iprod{(A^\prime)^{-1}a^\prime}{(A^\prime)^{-1}y}+2\left|(A^\prime)^{-1}
a^\prime\right|^2-\ln(\alpha^\prime)
\geq \vartheta^2\norm{y}_K^2>\vartheta^2. 
\]
As $\left|(A^\prime)^{-1}y\right|< 1$, letting $\vartheta$ tend to infinity, 
we obtain a contradiction. Thus, $A^\prime\ball{d}\subseteq K$.

Hence, $\det(A^\prime)\leq\det(A)$. On the other hand, $\alpha^\prime\leq 
\norm{f}=1$. The Proposition now easily follows from \eqref{eq:gaussian_volume}.
\end{proof}

\subsection{Approximation of a largest Gaussian by John 
\texorpdfstring{$s$}{s}-ellipsoids}
\begin{thm}\label{thm:s_infty_approx}
Let $f : \Red \to [0, \infty)$ be a proper log-concave function. Then the 
following hold.
\begin{enumerate}
\item\label{item:gaussianexists} 
There is a Gaussian density below $f$ if and only if
$\limsup\limits_{s \to \infty} \smeasure{\sellbody{f}} > 0.$
\item\label{item:gaussianislimit} 
If there is a Gaussian density below $f$, then 
$\lim\limits_{s \to \infty} \smeasure{\sellbody{f}} = \limsup\limits_{s \to 
\infty} \smeasure{\sellbody{f}} $ 
and there is a sequence $\{s_i\}_1^{\infty}$ of positive reals with 
$\lim\limits_{i \to \infty} s_i = \infty$ such that the
\jsfunctions{} $\selldense[s_i]{f}$  converge uniformly  to a  Gaussian density 
which is of maximal 
integral among those Gaussian densities that are below $f.$
\item\label{item:gaussianheightbound} 
If there is a Gaussian density below $f$, then 
any Gaussian density which is of maximal 
integral among those Gaussian densities that are below $f$ is of height at least 
$\norm{f}e^{-d}$.
\end{enumerate}
\end{thm}
In the rest of this subsection, we prove \Href{Theorem}{thm:s_infty_approx}.

We start by describing the limit of $s$-marginals of origin centered ellipsoids.
\begin{lemma}\label{lem:limit_of_balls}
Let  $\{s_i\}_1^{\infty}$ be a sequence of positive reals with 
$\lim\limits_{i \to \infty} s_i = \infty $, let $\{A_i\}_1^{\infty}$ be a 
sequence 
of positive definite operators such that $\lim\limits_{i \to \infty} 
\frac{A_i}{\norm{A_i}} = A,$ where $A$ is positive definite,  and  
the ellipsoids $\upthing{E}_{i},$ represented by 
$(A_i \oplus 1, 0),$ satisfy  
$\lim\limits_{i \to \infty} \smeasure[s_i]{\upthing{E}_i} =  \mu > 0.$ Then, 
the ${s_i}$-th power of the height functions  
$\left(\shf{\upthing{E}_i}\right)^{s_i}$
  converge uniformly to the Gaussian density $G[(A_\infty \oplus 1, 0)],$  where
\[
A_\infty = \frac{1}{\sqrt{\pi}} \left(\frac{\mu }{\det A }\right)^{1/d}
{A}.
\]  
\end{lemma}
\begin{proof}[Proof of Lemma~\ref{lem:limit_of_balls}]
The limit $\frac{A_i}{\norm{A_i}} \to A$ as $i \to \infty$ yields two 
properties,
\begin{equation}
\label{eq:convergence_det}
\lim\limits_{i \to \infty} \frac{\det{A_i}}{\norm{A_i}^d} = \det A, 
\end{equation}
and 
\begin{equation}
\label{eq:convergence_reverse_matrix}
\lim\limits_{i \to \infty} {\norm{A_i}} A_i^{-1} = A^{-1} . 
\end{equation}

By \eqref{eq:sfunc_of_ellipsoida} and \eqref{eq:volbs_limit_at_infinity}, we 
obtain 
\begin{gather*}
\mu = \lim\limits_{i \to \infty} \smeasure[s_i]{\upthing{E}_i} = 
\lim\limits_{i \to \infty} \volbs[s_i] \det A_i = 
\pi^{d/2} \lim\limits_{i \to \infty} \left(\frac{s_i}{2}
 \right)^{-d/2} \det A_i.
\end{gather*}
Combining this with \eqref{eq:convergence_det}, we get
\begin{equation}
\label{eq:convergense_s_norm_operator}
\lim\limits_{i \to \infty} \frac{s_i}{2}   \frac{1}{\norm{A_i}^2}= 
{\pi}\left(\frac{\det A}{\mu}\right)^{2/d}.
\end{equation}

We obtain that $\lim\limits_{i \to \infty} \norm{A_i}=\infty$, and hence by 
\eqref{eq:convergence_reverse_matrix}, the smallest eigenvalue of $A_i$ tends 
to 
infinity. 

It follows that for any fixed $\rho > 0$ and sufficiently large $i$, we have 
$\rho \ball{d}\subset A_i\ball{d}$, and hence,
\[
\left(\shf{\upthing{E}_i}(x)\right)^{s_i} = 
\left(1 - \iprod{A_i^{-1}x}{A_i^{-1}x}\right)^{s_i/2} =
\left(1 - 
\frac{\iprod{\norm{A_i}A_i^{-1}x}{\norm{A_i}A_i^{-1}x}}{\norm{A_i}^2}\right)^{
\norm{A_i}^2 \cdot \frac{s_i}{2}   \frac{1}{\norm{A_i}^2}}
\]
for all  $x \in \rho \ball{d}$.

By  \eqref{eq:convergense_s_norm_operator} and 
\eqref{eq:convergence_reverse_matrix}, for any $1 > \epsilon > 0,$ there exists 
$i_\epsilon$ such that 
the  inequalities 
\[
\left(1 - \frac{(1 + \epsilon)\iprod{A^{-1} x}{A^{-1} 
x}}{\norm{A_i}^2}\right)^{\norm{A_i}^2 \left[{\pi}\left(\frac{\det 
A}{\mu}\right)^{2/d}(1 + 
\epsilon)\right]}
\leq
  \left(\shf{\upthing{E}_i}(x)\right)^{s_i},
\]
\[
  \left(\shf{\upthing{E}_i}(x)\right)^{s_i}
  \leq
 \left(1 - \frac{(1 - \epsilon)\iprod{A^{-1} x}{A^{-1} 
x}}{\norm{A_i}^2}\right)^{\norm{A_i}^2 \left[{\pi}\left(\frac{\det 
A}{\mu}\right)^{2/d}(1 - 
\epsilon)\right]}.
\]
hold for all $x \in \rho \ball{d}$ and for all $i > i_\epsilon.$
Since $\lim\limits_{i \to \infty} \norm{A_i}=\infty$, this implies that the 
sequence of functions 
$\left\{\left(\shf{\upthing{E}_i}(x)\right)^{s_i}\right\}_{i=1}^{\infty}$  
converges uniformly to $g(x)=e^{- {\pi}\left(\frac{\det 
A}{\mu}\right)^{2/d}\iprod{A^{-1} x}{A^{-1} x}}$ on $ \rho \ball{d}.$ 
Since $\sup_{x\in\Red \setminus \rho \ball{d}}g(x)$ tends to zero as $\rho \to 
\infty$ and 
$\lim\limits_{i \to \infty} \smeasure[s_i]{\upthing{E}_i}=\int_{\Red}g$, 
 we conclude that  $\left\{\left(\shf{\upthing{E}_i}(x)\right)^{s_i}\right\}$ is 
uniformly convergent on $\Red$
This completes the proof 
of \Href{Lemma}{lem:limit_of_balls}.
\end{proof}
\begin{lemma}\label{lem:john_gaussian_approx}
Let $G$ be a Gaussian density. Then, 
the \jsfunctions{} $\selldense{G}$ converge  uniformly to $G$ on $\Red$ as 
$s \to \infty.$
\end{lemma}
In order to prove Lemma~\ref{lem:john_gaussian_approx}, we assume that
$G(x) = e^{-{|x|^2}/{2}}.$  First, we relax the condition and prove that it 
suffices to approximate $G$ by any sequence of suitable height functions of 
$d$-ellipsoids.
\begin{claim}\label{claim:convenient_approx_gauss_john}
If there is a function $c : [1, \infty) \to [1, \infty)$  such that  
\begin{equation}
\label{eq:convergence_j_standard_gaussian}
\lim\limits_{s \to \infty}  \int_{\R^d} \shf{(c(s)I \oplus 1,0)}^{s} = 
\int_{\R^d} G = {(2\pi)}^{d/2},
\end{equation}
 and  
$
 \shf{(c(s)I \oplus 1,0)}^{s} \leq G
$
{for all}  $s \geq 1,$
then the functions $\selldense{G}$ converge  uniformly to $G$ on $\Red$ as 
$s \to \infty.$
\end{claim}
\begin{proof}[Proof of Claim~\ref{claim:convenient_approx_gauss_john}]
By Theorem \ref{thm:johnunicity}, the \jsfunction{} $\selldense{G}$ of $G$ 
exists and is unique for any positive $s.$  By symmetry, we see that 
$\selldense{G}$ is of the form 
$\shf{\left(\beta(s) I \oplus \alpha(s), 0 \right)}^s,$
 where $\beta \st [1, \infty) \to (0, \infty)$ and 
$\alpha \st [1, \infty) \to (0, 1).$ By 
\eqref{eq:convergence_j_standard_gaussian}, we obtain that
\[
\lim\limits_{s \to \infty}  \int_{\R^d} \selldense{G} = \int_{\R^d} G.
\]
This implies that $\alpha(s) \to 1$ as $s \to \infty.$
Hence,  the functions 
$\selldense{G}= \shf{\left(\beta(s) I \oplus \alpha(s), 0 \right)}^{s}$ 
converge  uniformly to the same function as 
the functions $\shf{\left(\beta(s) I \oplus 1, 0 \right)}^s$ as 
$s\rightarrow\infty$ (if the latter sequence converges). However, by Lemma 
\ref{lem:limit_of_balls}, the functions 
 $\shf{\left(\beta(s) I \oplus 1, 0 \right)}^s$ converge uniformly to $G$ as $s 
\to \infty.$
\end{proof}
It is not hard to find a suitable function $c(s).$
\begin{claim}\label{claim:explicit_approx_gauss_john}
Let $c(s) = \sqrt{s}.$ Then,
 $
\shf{(c(s)I \oplus 1,0)}^s \leq G
$
{for all}  $s \geq 1,$ and identity \eqref{eq:convergence_j_standard_gaussian} 
holds.
\end{claim}
\begin{proof}[Proof of Claim~\ref{claim:explicit_approx_gauss_john}]
Identity \eqref{eq:convergence_j_standard_gaussian} is an immediate consequence 
of \eqref{eq:sfunc_of_ellipsoida} and \eqref{eq:volbs_limit_at_infinity}.

 Inequality $ \shf{(c(s)I \oplus 1,0)}^s \leq G
$ is purely technical.  
By a routine calculation, for any $x \in \R^d,$ we have
\[
\lim\limits_{s \to \infty} \shf{(c(s)I \oplus 1,0)}^{s}(x) {=} 
\lim\limits_{s \to \infty} \left(1 - \frac{|x|^2}{s}\right)^{s/2}
= G(x).
\]
As is easily seen, 
$\shf{(c(s)I \oplus 1,0)}^s(x)$ is an increasing function of $s\in [1,\infty)$
for any fixed $x \in \R^d.$
\end{proof}

\Href{Lemma}{lem:john_gaussian_approx} follows from Claims 
\ref{claim:convenient_approx_gauss_john} and 
\ref{claim:explicit_approx_gauss_john}.

\begin{lemma}\label{lem:gaussian_part_limit}
If $\limsup\limits_{s \to \infty} \smeasure{\sellbody{f}} > 0,$ then there 
exists a sequence $\{s_i\}_1^{\infty}$ of positive reals with 
$\lim\limits_{i \to \infty} s_i = \infty$ such that the  
\jsfunctions{} $\selldense[s_i]{f}$ converge uniformly on $\Red$ to a Gaussian 
density, which we denote by $G[(A_\infty \oplus \alpha_\infty, a_\infty)]$, 
which is below $f$ and is of maximal integral among Gaussian densities below 
$f$.
Moreover, we have
\[
  \int_{\Red} G[(A_\infty \oplus \alpha_\infty, a_\infty)]= 
 \limsup\limits_{s \to \infty} \smeasure{\sellbody{f}},
 \text{ and } \alpha_\infty\in[e^{-d}\norm{f},\norm{f}].
\]
\end{lemma}

In the proof of Lemma~\ref{lem:gaussian_part_limit}, we will need the following 
immediate consequence of \Href{Lemma}{lem:boundedness} and 
\eqref{eq:volbs_limit_at_infinity}.
\begin{claim}\label{claim:norm_bound}
Let $f:\Red\to[0,\infty)$ be a proper log-concave function, and $\delta, 
s_0>0$. Then there exist  $\rho_1, \rho_2 >0 $ such that 
for any $ s \geq s_0$, if 
$\upthing{E}=(A\oplus\alpha)\ball{d+1}+a$, where $(A\oplus\alpha,a)\in\ellips$, 
is a \dsymm{} ellipsoid in $\Redp$ with
$\upthing{E}\subseteq \slift{f}$ and 
$\smeasure{{\upthing{E}}}\geq\delta$, then we have
\begin{equation}
\label{eq:comparison_operator_big_s}
\rho_1  I \prec \frac{A}{\sqrt{s}} \prec \rho_2 I.
\end{equation}
\end{claim}

\begin{proof}[Proof of Lemma~\ref{lem:gaussian_part_limit}]
Let $(A_s \oplus \alpha_s, a_s)$ represent $\sellbody{f}.$
By  \Href{Lemma}{lem:height}, we have that $\norm{\selldense{f}}$ belongs to 
the interval $[e^{-d}\norm{f},\norm{f}]$. Thus, $[f\geq 
e^{-d}\norm{f}]\supseteq\{ a_s \}_{s > 0}$. Since $f$ is a proper log-concave 
function, the set $[f\geq e^{-d}\norm{f}]$ is a bounded subset of $\Red$, and 
thus, so is $\{ a_s \}_{s > 0}$.
Thus, there exists a sequence   $\{s_i\}_1^{\infty}$ with 
$\lim\limits_{i \to \infty} s_i = \infty$ such that 
\begin{equation}
\label{eq:limits_s_infty}
 \smeasure[s_i]{\sellbody[s_i]{f}} \to  \limsup\limits_{s \to \infty} 
\smeasure{\sellbody{f}}, \,
 \frac{A_{s_i}}{\norm{A_{s_i}}} \to  A , \;
 \norm{\selldense[s_i]{f}} \to \alpha_\infty > 0
\quad \text{and} \quad a_{s_i}  \to a_\infty
\end{equation}
{for some positive semidefinite matrix $A\in\Re^{d\times d}$, an 
$\alpha_\infty>0$ and $a_\infty\in\Red$,}
as $i$ tends to $\infty.$

\Href{Claim}{claim:norm_bound} implies that $A$ is positive definite.
Hence by \eqref{eq:limits_s_infty}, we may apply
\Href{Lemma}{lem:limit_of_balls}
to obtain that the sequence $\left\{\selldense[s_i]{f}\right\}_{i=1}^{\infty}$ 
converges uniformly to  the Gaussian density
$G[(A_\infty \oplus \alpha_\infty, a_\infty)]$, where
\[
A_\infty = \frac{1}{\sqrt{\pi}} \left(\frac{ \limsup\limits_{s \to \infty} 
\smeasure{\sellbody{f}}}{ \det A }\right)^{1/d}
{A}.
\] 
Clearly, $G[(A_\infty \oplus \alpha_\infty, a_\infty)] \leq f$ and, by the 
uniform convergence,
\[
\lim_{i\rightarrow\infty}\smeasure[s_i]{\sellbody[s_i]{f}} =
\limsup_{s\rightarrow\infty}\smeasure[s]{\sellbody[s]{f}} =
\int_{\Red} G[(A_\infty \oplus \alpha_\infty, a_\infty)].
\]
The latter implies that there is no Gaussian density below  $f$ with the 
integral strictly greater than $\int_{\Red} G[(A_\infty \oplus \alpha_\infty, 
a_\infty)],$ since, by  \Href{Lemma}{lem:john_gaussian_approx}, any 
Gaussian density $G^{\prime}$ is 
the limit of $\selldense{G^{\prime}}$ as $s \to \infty.$ 
\end{proof}

\begin{proof}[Proof of \Href{Theorem}{thm:s_infty_approx}]
First, assume that there is a Gaussian density $G$ below $f$. Then, by 
\Href{Lemma}{lem:john_gaussian_approx}, 
\[
\limsup\limits_{s \to \infty} \smeasure{\sellbody{f}}= \limsup\limits_{s \to 
\infty} \int_{\Red} \selldense{f}\geq\limsup\limits_{s \to \infty} \int_{\Red} 
\selldense{G}= 
\int_{\Red} G> 0.
\]
The converse in part \eqref{item:gaussianexists} follows from 
\Href{Lemma}{lem:gaussian_part_limit}.

To prove part \eqref{item:gaussianislimit}, assume again that there is a 
Gaussian density below $f$. 
By part \eqref{item:gaussianexists}, we may apply 
\Href{Lemma}{lem:gaussian_part_limit} and obtain a Gaussian density $G[(A_\infty 
\oplus \alpha_\infty, a_\infty)]$ with all the desired properties. 
We need to verify only that the limit in part \eqref{item:gaussianislimit} 
exists and is equal to the $\limsup$.
We have by \Href{Lemma}{lem:john_gaussian_approx} that
\[
 \smeasure{\sellbody{f}} \geq  
 \smeasure{\sellbody{G[(A_\infty \oplus \alpha_\infty, a_\infty)]}} 
\xrightarrow{s \to \infty}
 \int_{\Red} G[(A_\infty \oplus \alpha_\infty, a_\infty)]= 
 \limsup\limits_{s \to \infty} \smeasure{\sellbody{f}}, 
 \]
and hence, the limit $\lim\limits_{s \to \infty} \smeasure{\sellbody{f}}$ exists 
completing the proof of part \eqref{item:gaussianislimit}.

Part \eqref{item:gaussianheightbound} follows immediately from 
\Href{Lemma}{lem:gaussian_part_limit} and 
\Href{Theorem}{thm:infinity_ellipsoid}.
\end{proof}

\section{The Helly type result --- Proof of 
Theorem~\ref{thm:BKP}}\label{sec:BKP}

In this section, we prove \Href{Theorem}{thm:BKP}.

\subsection{Assumption: the functions are supported on 
\texorpdfstring{$\Red$}{R\string^d}}\label{subsec:assumptionsupport}
We claim that we may assume that 
the support of each $f_i$ is $\Red$. 
Indeed, any log-concave function can be approximated in the $L_1$-norm by 
log-concave functions whose support is $\Red$. 
Recall that $f_\sigma$ denotes the pointwise minimum of functions
$\{f_i\}_{i \in \sigma}$ for $\sigma\subseteq[n]$.
We may approximate each function 
so that the $f_\sigma$ are also all well approximated. One way to achieve this 
is to 
take the Asplund sum $\loginfconv{f_i}{( e^{-\delta |x|^2})}$ for a 
sufficiently large $\delta>0$ (see \Href{Section}{sec:asplund}).
\subsection{Assumption: John position}\label{subsec:assumptionjohn}
Consider the $s$-lifting of our functions with $s=1$. Clearly, the 
$s$-lifting of a pointwise minimum of a family of functions is the intersection 
of the $s$-liftings of the functions.

From our assumption in  \Href{Subsection}{subsec:assumptionsupport}, it follows 
that 
$\int_{\Red} 
f>0$. By applying a linear transformation on $\Red$, we may assume 
that, with $s=1$, the largest $s$-volume ellipsoid in the $s$-lifting 
$\slift[1]{f}$ of $f$ is $\ball{d+1}\subset\slift[1]{f}.$

By \Href{Theorem}{thm:johncond}, there are contact points 
$\upthing{u}_1,\ldots,\upthing{u}_k\in\bd{\ball{d+1}}\cap\bd{\slift[1]{f}}$, 
and positive weights $c_1,\ldots,c_k$ satisfying \eqref{eq:johncond} with $s=1$.
For each $j\in[k]$, we denote by $u_j$ the orthogonal projection of the contact 
point $\upthing{u}_j$ onto $\Red$ and by $w_j=\sqrt{1-|u_j|^2}$.

\newcommand{\contactindices}{\eta}
\subsection{Reduction of the problem to finding \texorpdfstring{$P$ and 
$\contactindices$}{P and eta}}\label{subsec:reduction}

\begin{claim}\label{claim:needpandeta}
With the assumptions in Subsections~\ref{subsec:assumptionsupport} and 
\ref{subsec:assumptionjohn}, we can find a set of indices 
$\contactindices\in\binom{[k]}{\leq 2d+1}$ and an origin-symmetric convex body 
$P$ in $\Red$ with the following two properties.
\begin{equation}\label{eq:volumeboundforP}
 \vol{d}{P}\leq 8 \cdot 4^d \cdot d^{d} (d+2)^d  \left(\vol{d} \ball{d} \right)^2
\end{equation}
and
\begin{equation}\label{eq:iprodboundforP}
\norm{x}_{P}\leq
 \max\big\{\iprod{x}{u_j}\st j\in \contactindices\big\}\;\;\text{ for 
every } x\in\Red,
\end{equation}
where $\norm{\cdot}_{P}$ is the gauge function of $P$, that is, 
$\norm{x}_{P}=\inf\{\lambda>0\st x\in\lambda P\}$.
\end{claim}

We will prove \Href{Claim}{claim:needpandeta} in 
\Href{Subsection}{subsec:findingpandeta}.

\emph{
In the present subsection, we show that \Href{Claim}{claim:needpandeta} yields 
the existence of the desired index set $\sigma\in\binom{[n]}{\leq \hellyno}$ 
that satisfies \eqref{eq:BKPgoal}.
}

The \emph{polar} of a set $K$ in $\Re^n$ is defined by
$K^\circ =\left\{ p \in \R^n \st \iprod{y}{p} \leq 1
\text{ for all } y \in K \right\}.$
Set $T=\{u_j\st j\in \contactindices\}$. 
It is easy to see that \eqref{eq:iprodboundforP} is equivalent to
\begin{equation}\label{eq:iprodboundforPcontainment}
 T^{\circ}\subseteq P.
\end{equation}
Notice that
\begin{equation}\label{eq:outsidepositive}
\text{ for any }x\in\Red\setminus T^{\circ}, \text{ there is } j\in 
\contactindices 
\text{ such that } \iprod{u_j}{x-u_j}\geq 0.
\end{equation}

We will split the integral in \eqref{eq:BKPgoal} into two parts: the integral 
on $\Red\setminus T^\circ$ and the integral on $T^\circ$.

First, we find a set $\sigma_1$ of indices in $[n]$ that will help us bound the 
integral in \eqref{eq:BKPgoal} on $\Red\setminus T^\circ$.

Fix a $j\in \contactindices$. Since $\upthing{u}_j\in\bd{\slift[1]{f}}$, there 
is an 
index $i(j)\in[n]$ such that $\upthing{u}_j\in\bd{\slift[1]{f}_{i(j)}}$. Let 
$\sigma_1$ be the set of these indices, that is, $\sigma_1=\{i(j)\st 
j\in \contactindices\}$.

By \eqref{eq:touchingballbound}, for each $j\in \contactindices$, we have
\begin{equation}\label{eq:touchingballboundBKP}
 f_{i(j)}(x)\leq 
 w_j e^{- \frac{1}{w_j^2}\iprod{u_j}{x-u_j}}\leq
 e^{- \frac{1}{w_j^2}\iprod{u_j}{x-u_j}}
\end{equation}
 for all $x\in\Red.$
 
Next, we find a set $\sigma_2$ of indices in $[n]$ that will help us bound the 
integral in 
\eqref{eq:BKPgoal} on $T^\circ$.

It is easy to see that there is a $\sigma_2\in\binom{[n]}{\leq d+1}$ such that 
$\norm{f}=\norm{f_{\sigma_2}}$. Indeed,  for any
$i\in[n]$, consider the following convex set in $\Red$:
$[f_i>\norm{f}]$. By the definition of $f$, the 
intersection of these $n$ convex sets in $\Red$ is empty. Helly's theorem 
yields the existence of $\sigma_2$.

We combine the two index sets: let $\sigma=\sigma_1\cup \sigma_2$. Clearly, 
$\sigma$ is of size at most $\hellyno$. We need to show that $\sigma$ 
satisfies \eqref{eq:BKPgoal}. 

Note that by \eqref{eq:height} and Assumption~\ref{subsec:assumptionjohn}, we 
have $\norm{f_{\sigma_2}}=\norm{f}\leq e^d$. Hence,

\begin{equation*}
\int_{\Red} f_\sigma \leq 
\int_{T^{\circ}} \norm{f_\sigma} +
\int_{\Red\setminus T^{\circ}} f_\sigma\leq
\int_{T^{\circ}} e^d +
\int_{\Red\setminus T^{\circ}} 
f_\sigma
\stackrel{\eqref{eq:iprodboundforPcontainment}}{\leq}
 e^d\vol{d} P +
\int_{\Red\setminus T^{\circ}} 
f_{\sigma_1}
\end{equation*}
Next, we bound the second summand using the tail bound 
\eqref{eq:touchingballboundBKP}.
\[
\int_{\Red\setminus T^{\circ}} 
f_{\sigma_1}
\stackrel{\eqref{eq:touchingballboundBKP}}{\leq}
 \int_{\Red\setminus T^{\circ}} 
  \exp\left(-\max\left\{  \frac{1}{w_j^2}\iprod{u_j}{x-u_j}
   \st j\in 
\contactindices\right\}\right)\stackrel{\eqref{eq:outsidepositive}}{\leq}
\]\[
 \int_{\Red\setminus T^{\circ}} 
  \exp\left(-\max\left\{\iprod{u_j}{x-u_j}
   \st j\in \contactindices\right\}\right)\leq
 e\int_{\Red\setminus T^{\circ}} 
  \exp\left(-\max\left\{\iprod{u_j}{x}
   \st j\in 
\contactindices\right\}\right).
\]
By property \eqref{eq:iprodboundforP}, the latter is at most
\[
 e\int_{\Red\setminus T^{\circ}} 
  \exp\left(-\norm{x}_P\right)\leq
 e\int_{\Red} 
  \exp\left(-\norm{x}_P\right)=
e\cdot d! \vol{d} P.
\]
Hence,
\[
\int_{\Red} f_\sigma \leq  \left(e^d + e \cdot d!\right) \vol{d} P \leq 10 
\cdot d^{d-1} \vol{d} P.
\]
Using, the fact that 
\[\vol{d} \ball{d} \leq  d\vol{d+1} \ball{d+1} \leq d\int_{\Red} f,\] and  inequality
\eqref{eq:volumeboundforP} here, we 
obtain 
\[
\int_{\Red} f_\sigma \leq  
80 \cdot 4^d  \cdot d^{2d} (d+2)^d   \vol{d} \ball{d}  \cdot \int_{\Red} f.
\]
This inequality  directly implies  inequality \eqref{eq:BKPgoal}  in the case $d=1$. 
Consider $d \geq 2.$ Then, since  $d+2 \leq 2d$ and $\vol{d} \ball{d} \leq 10^d d^{-d/2},$ we conclude that 
\[
\int_{\Red} f_\sigma \leq  
80 \cdot 80^d  \cdot d^{5d/2}  \int_{\Red} f \leq  \left(100 d\right)^{5d/2}  \int_{\Red} f,
\]
completing the proof of inequality \eqref{eq:BKPgoal}.
\subsection{The Dvoretzky-Rogers lemma}
One key tool in proving \Href{Claim}{claim:needpandeta} is the Dvoretzky--Rogers 
lemma \cite{dvoretzky1950absolute}.

\begin{lemma}[Dvoretzky-Rogers lemma]\label{lem:dvoro}
 Assume that the points 
$\upthing{u}_1,\ldots,\upthing{u}_k\in\bd{\ball{d+1}}$, 
satisfy \eqref{eq:johncond} for $s=1$ with some positive weights 
$c_1,\ldots,c_k$.
Then there is a
sequence $j_1,\ldots,j_{d+1}$ of $d+1$ distinct indices in $[k]$ such that 
\begin{equation*}\label{eq:dvoro}  
\mathrm{dist}\big(\upthing{u}_{j_t},\mathrm{span}\{\upthing{u}_{j_1},\ldots,
\upthing{u}_{j_{t-1}}\}\big)\geq
  \sqrt{\frac{d-t+2}{d {+1}}}
  \;\;\mbox{ for all }t=2,\ldots,d+1,\;\;
\end{equation*}
where $\mathrm{dist}$ denotes the shortest Euclidean distance between a vector 
and a subspace.
\end{lemma}
It follows immediately that the determinant of the $(d+1)\times(d+1)$ matrix 
with columns $\upthing{u}_{j_1},\ldots\upthing{u}_{j_{d+1}}$ is at least 
\begin{equation}\label{eq:dvorodet}
 \left|\det \left[\upthing{u}_{j_1},\ldots\upthing{u}_{j_{d+1}}\right]\right|
 \geq 
 \frac{\sqrt{(d+1)!}}{(d+1)^{(d+1)/2}}.
\end{equation}

\subsection{Finding \texorpdfstring{$P$ and $\contactindices$}{P and 
eta}}\label{subsec:findingpandeta}
In this subsection, we prove \Href{Claim}{claim:needpandeta}, that is, we
show that with the assumptions in 
Subsections~\ref{subsec:assumptionsupport} and \ref{subsec:assumptionjohn}, 
there is an origin symmetric convex body $P$ and a set of indices 
$\contactindices\in\binom{[k]}{\leq 2d+1}$ satisfying 
\eqref{eq:volumeboundforP} and \eqref{eq:iprodboundforP}. Once it is shown, by 
\Href{Subsection}{subsec:reduction}, the proof of \Href{Theorem}{thm:BKP} is 
complete.

The proof in this section follows very closely the proof of the main result in 
\cite{Nas16} as refined by Brazitikos in \cite{Bra17}.

Let $\contactindices_1\in\binom{[k]}{d+1}$ be the set of $d+1$ indices in $[k]$ 
given by \Href{Lemma}{lem:dvoro}, and let $\upthing{\Delta}$ be the simplex 
$\upthing{\Delta}=\conv{\{\upthing{u}_j\st j\in 
\contactindices_1\}\cup\{0\}}$ in $\Redp$.
Let $\upthing{z}=\frac{\sum_{j\in \contactindices_1}\upthing{u}_j}{d+1}$ denote 
the centroid of $\upthing{\Delta}$, and $\upthing{P}_1$ denote the intersection 
of 
$\upthing{\Delta}$ and its reflection about $\upthing{z}$, that is, 
$\upthing{P}_1=\upthing{\Delta}\cap(2\upthing{z}-\upthing{\Delta})$, a 
polytope which is centrally symmetric about $\upthing{z}$. It is well known 
\cite[Corollary 3]{milman2000entropy} (see also \cite[Section 
4.3.5]{aubrun2017alice}),
that $\vol{d+1}{\upthing{P}_1}\geq 2^{-(d+1)}\vol{d+1}{\upthing{\Delta}}$, 
and hence, by \eqref{eq:dvorodet}, we have 
$$
 \vol{d+1}{\upthing{P}_1}\geq 2^{-(d+1)}
 \frac{\left|\det[\upthing{u}_j\st j\in \contactindices_1]\right|}{(d+1)!}
 \geq  \frac{1}{2^{d+1}\sqrt{(d+1)!}(d+1)^{(d+1)/2}}
$$

Let $P_1$ denote the orthogonal projection of $\upthing{P}_1$ onto $\Red$. 
Since 
$\upthing{P}_1\subset P_1\times[-1,1]$, we have that
\begin{equation}\label{eq:poneisbig}
 \vol{d} P_1
 \geq  \frac{1}{2^{d+2}\sqrt{(d+1)!}(d+1)^{(d+1)/2}}.
\end{equation}

Moreover, $P_1$ is symmetric about the orthogonal projection $z$ of 
$\upthing{z}$ onto $\Red$.

Let $\upthing{Q}$ denote the convex hull of the contact points,
$\upthing{Q}={\conv{\bd{\slift{f}}\cap\bd{\ball{d+1}}}}$, and $Q$ denote the 
orthogonal projection of $\upthing{Q}$ onto $\Red$.
As a well known consequence of \eqref{eq:johncond} for $s=1$ 
\cite{ball1997elementary}, we have
$\frac{1}{d+1}\ball{d+1}\subset \upthing{Q}$, and hence, 
$\frac{1}{d+1}\ball{d}\subset Q$.

Let $\ell$ be the ray in $\Red$ emanating from the origin in the 
direction of the vector $-z$, and let $y$ be the 
point of intersection of $\ell$ with the boundary (in $\Red$) of $Q$, that is, 
$\{y\}=\ell\cap\bd{Q}$. Now, 
$\frac{1}{d+1}\ball{d}\subset Q$ yields that $|y|\geq 1/(d+1)$.

We apply a contraction with center $y$ and ratio 
$\lambda=\frac{|y|}{|y-z|}$ on 
$P_1$ to obtain the polytope $P_2$. Clearly, 
$P_2$ is a convex polytope in $\Red$ which is symmetric about the origin. 
Furthermore,
\begin{equation}\label{eq:lambdanotsmall}
 \lambda=\frac{|y|}{|y-z|}\geq 
 \frac{|y|}{1+|y|}\geq\frac{1}{d+2}.
\end{equation}

Let $P$ be the polar $P=P_2^{\circ}$ of $P_2$ taken in $\Red$.
By the Santal\'o inequality \cite[Theorem 9.5]{GruberBook}, we obtain
\begin{equation*}
 \vol{d} P \leq
 \frac{\left(\vol{d}\ball{d}\right)^2}{\vol{d} P_2 }=
 \frac{\left(\vol{d}\ball{d}\right)^2}{\lambda^d\vol{d} P_1 },
\end{equation*}
which, by \eqref{eq:poneisbig}, the inequality $d+1 \leq 2d$ and \eqref{eq:lambdanotsmall}, yields 
that $P$ satisfies \eqref{eq:volumeboundforP}.

To complete the proof, we need to find $\contactindices\in\binom{[k]}{\leq 
2d+1}$ such that $P$ and $\contactindices$ satisfy 
\eqref{eq:iprodboundforP}.

Since $y$ is on $\bd{Q}$, by Carath\'eodory's theorem, 
$y$ is in the convex hull 
of some subset of at most $d$ vertices of $Q$. Let this subset be 
$\{u_j\st j\in \contactindices_2\}$, where 
$\contactindices_2\in\binom{[k]}{\leq d}$.

We set $\contactindices=\contactindices_1\cup \contactindices_2$, and
claim that $P$ and $\contactindices$ satisfy \eqref{eq:iprodboundforP}.

Indeed, since 
$
 P_2\subseteq 
 \conv{\{u_j\st j\in\contactindices_1\}\cup\{y\}}
$ and
$
 y\in \conv{\{u_j\st j\in\contactindices_2\}}
$, we have 
\begin{equation*}
 P_2\subseteq 
 \conv{\{u_j\st j\in\contactindices\}}.
\end{equation*}

Taking the polar of both sides in $\Red$, we obtain 
$
P\supseteq \{u_j\st j\in\contactindices\}^{\circ},
$ which is equivalent to \eqref{eq:iprodboundforP}.

Thus, $P$ and $\contactindices$ satisfy \eqref{eq:volumeboundforP} and 
\eqref{eq:iprodboundforP}, and hence, the proof of \Href{Theorem}{thm:BKP} is 
complete.

\subsection{Lower bound on the Helly number}
\label{sec:BKPlowerbound}

The number of functions selected in \Href{Theorem}{thm:BKP} is $3d+2$. In this 
subsection, we show that it cannot be decreased to $2d$.
In fact, 
for any dimension $d$ and any $\Delta > 0$, we give an example of $2d+1$ 
log-concave 
functions $f_1,\ldots,f_{2d+1}$ such that $\int f_{[n]}=2^d$, but for any 
$I\in\binom{[2d+1]}{\leq2d}$, the integral is 
$\int f_{I}  > \Delta$.
Our example is a simple extension of the standard one (the $2d$ supporting 
half-spaces of a cube) for convex sets. 

Set
\[
 \phi(t)=\begin{cases}
           0 ,&\text{ if } t <  0\\
          e^{\Delta},&\text{ otherwise}.
         \end{cases}
\]
Clearly, $\phi$ is upper semi-continuous.
Let $e_1,\ldots,e_{d}$ 
denote the standard basis in $\Red$, and for each $i\in[d]$, define the 
functions
$f_i(x)=\phi(\iprod{e_i}{x+e_i})$ and 
$f_{d+i} = \phi(-\iprod{e_i}{x-e_i})$, and let $f_{2d+1}=1$.
These functions are proper log-concave functions. The bounds on the integrals 
are easy.

\subsection*{Acknowledgement}
G.I. was supported  by the Ministry of Education and Science of the Russian 
Federation in the framework of MegaGrant no 075-15-2019-1926.
M.N. was supported by the National Research, Development and Innovation Fund (NRDI) grants K119670 and KKP-133864 as
well as the Bolyai Scholarship of the Hungarian Academy of Sciences and the New National
Excellence Programme and the TKP2020-NKA-06 program provided by the NRDI.

\bibliographystyle{alpha}
\bibliography{biblio}
\end{document}